\documentclass[11pt]{article}
\usepackage[utf8]{inputenc}

\usepackage[title]{appendix}
\usepackage{epsfig,epsf,fancybox}
\usepackage{amsmath}
\usepackage{mathrsfs}
\usepackage{amssymb}
\usepackage{graphicx}
\usepackage{color}
\usepackage{multirow}
\usepackage{paralist}
\usepackage{verbatim}
\usepackage{galois}
\usepackage{algorithm}
\usepackage{algorithmic}
\usepackage{boxedminipage}
\usepackage{booktabs}
\usepackage{accents}
\usepackage{stmaryrd}
\usepackage{subfig}
\usepackage{epstopdf}
\usepackage{amsthm}
\usepackage{cases}  
\usepackage[top=1in,bottom=1.2in,left=1in,right=1in,xetex]{geometry}
\usepackage[pdfborder={0 0 0},colorlinks=true,linkcolor=blue,CJKbookmarks=true]{hyperref}
\usepackage{enumerate}
\usepackage{amsrefs}

\newcommand\blfootnote[1]{%
	\begingroup
	\renewcommand\thefootnote{}\footnote{#1}%
	\addtocounter{footnote}{-1}%
	\endgroup
}

\allowdisplaybreaks[4]

\newtheorem{theorem}{Theorem}[section]
\newtheorem{lemma}[theorem]{Lemma}
\newtheorem{assumption}[theorem]{Assumption}
\newtheorem{proposition}[theorem]{Proposition}
\newtheorem{definition}[theorem]{Definition}
\newtheorem{remark}[theorem]{Remark}

\title{Mild Solution of Semilinear SPDEs with Young Drifts\footnotemark[1]}
\author{Jiahao Liang\footnotemark[2] \ and Shanjian Tang\footnotemark[3]} 

\begin{document}

\maketitle

\pagenumbering{arabic}

\begin{abstract}
In this paper, we study a semilinear SPDE with a linear Young drift 
$du_{t}=Lu_{t}dt+f\left(t, u_{t}\right)dt+\left(G_{t}u_{t}+g_{t}\right)d\eta_{t}+h\left(t, u_{t}\right)dW_{t}$, 
where $L$ is the generator of an analytical semigroup, 
$\eta$ is an $\alpha$-Hölder continuous path with $\alpha \in \left(1/2, 1\right)$ 
and $W$ is a Brownian motion. 
After establishing through two different approaches the Young convolution integrals for stochastic integrands, 
we introduce the corresponding definition of mild solutions and continuous mild solutions, 
and give via a fixed-point argument the existence and uniqueness of the (continuous) mild solution  
under suitable conditions.  
\end{abstract}

\renewcommand{\thefootnote}{\fnsymbol{footnote}}

\blfootnote{\textit{MSC2020 subject classifications}: 60H15, 60L50. }
\footnotetext[1]{This work was partially supported by National Natural Science Foundation of China (Grants No. 12031009),  Key Laboratory of Mathematics for Nonlinear Sciences (Ministry of Education), and Shanghai Key Laboratory for Contemporary Applied Mathematics, Fudan University, Shanghai 200433, China.}
\footnotetext[2]{School of Mathematical Sciences, Fudan University, Shanghai 200433, China. E-mail: jhliang20@fudan.edu.cn.}
\footnotetext[3]{Institute of Mathematical Finance and Department of Finance and Control Sciences, School of Mathematical Sciences, Fudan University, Shanghai 200433, P. R. China. Email: sjtang@fudan.edu.cn.}

\section{Introduction}
\label{Sec1}
The Young integral introduced by Young \cite{Y36} extends the Riemann-Stieltjes integral 
$\int Yd\eta$ when $\eta$ and $Y$ are continuous and have finite 
$\frac{1}{\alpha}$-variation and $\frac{1}{\beta}$-variation respectively 
(equivalently through a time transformation, 
$\eta$ and $Y$ are $\alpha$-Hölder and $\beta$-Hölder continuous respectively) 
with $\alpha+\beta > 1$. 
On this basis, Lyons \cites{L94, L98} and Gubinelli \cite{G04} develop a more general theory 
of rough integrals and rough differential equations (RDEs).\\
\indent
PDEs driven by irregular paths have been well-studied. 
One of the important approaches is to study mild solutions of these PDEs. 
Mild solutions of the semilinear Young PDE 
\begin{equation} \label{ypde}
	\left\{
		\begin{aligned}
			&du_{t}=\left[Lu_{t}+f\left(u_{t}\right)\right]dt+g\left(u_{t}\right)d\eta_{t}, \quad t \in \left(0, T\right],\\
			&u_{0}=\xi
		\end{aligned}
	\right.	
\end{equation}
were first studied by Gubinelli et al. \cite{GLT06}, 
where $L$ is the generator of an analytical semigroup $\left(S_{t}\right)_{t \geq 0}$ 
on a Hilbert space $\mathcal{H}$ and 
$\eta \in C^{\alpha}\left(\left[0, T\right], \mathbb{R}^{e}\right)$ for some $\alpha \in \left(\frac{1}{2}, 1\right)$. 
Gubinelli and Tindel \cite{GT10} obtain the existence and uniqueness of the mild solution 
of the semilinear rough PDE 
\begin{equation} \label{rpde}
	\left\{
		\begin{aligned}
			&du_{t}=\left[Lu_{t}+f\left(u_{t}\right)\right]dt+g\left(u_{t}\right)d\mathbf{X}_{t}, \quad t \in \left(0, T\right],\\
			&u_{0}=\xi,
		\end{aligned}
	\right.	
\end{equation}
where $g$ is linear or polynomial and $\mathbf{X}$ is a $\sigma$-Hölder rough path 
with $\sigma \in \left(\frac{1}{3}, \frac{1}{2}\right]$. 
Then Deya et al. \cite{DGT12} study the rough PDE \eqref{rpde} for general $g$. 
They obtain the local existence and uniqueness of the mild solution and 
construct a global mild solution under stronger regularity assumptions. 
After that, the rough PDE \eqref{rpde} has also been studied by 
Gerasimovičs and Hairer \cite{GH19} and Hesse and Neamţu \cites{HN19, HN20, HeN22}. 
It has also been extended to non-autonomous semilinear rough PDEs by Gerasimovičs et al. \cite{GHN21} 
and quasilinear rough PDEs by Hocquet and Neamţu \cite{HoN22}. 
On the other hand, Addona et al. \cites{ALT22-1, ALT22-2} study the smoothness of the mild solution 
of the Young PDE \eqref{ypde} and reduce the regularity requirement on the initial datum $\xi$.\\
\indent
In this paper, we consider the semilinear SPDE with a linear Young drift (Young SPDE)
\begin{equation} \label{yspdei}
	\left\{
		\begin{aligned}
			&du_{t}=\left[Lu_{t}+f\left(t, u_{t}\right)\right]dt+\left(G_{t}u_{t}+g_{t}\right)d\eta_{t}+h\left(t, u_{t}\right)dW_{t}, \quad t \in \left(0, T\right],\\
			&u_{0}=\xi,
		\end{aligned}
	\right.	
\end{equation}
where $L$ and $\eta$ are the same as in \eqref{ypde}, $W$ is a standard Brownian motion and 
coefficients $f, G, g$ and $h$ are random and time-varying. 
It connects to the SPDE driven by $W$ and an independent fractional Brownian motion 
$B^{H}$ with Hurst parameter $H \in \left(\frac{1}{2}, 1\right)$ 
\begin{equation*}
	\left\{
		\begin{aligned}
			&du_{t}=\left[Lu_{t}+f\left(t, u_{t}\right)\right]dt+\left(G_{t}u_{t}+g_{t}\right)dB_{t}^{H}+h\left(t, u_{t}\right)dW_{t}, \quad t \in \left(0, T\right],\\
			&u_{0}=\xi.
		\end{aligned}
	\right.	
\end{equation*}
Naturally, the Young SPDE \eqref{yspdei} can be formulated in a mild form
\begin{equation} \label{mf}
	u_{t}=S_{t}\xi+\int_{0}^{t}S_{t-r}f\left(r, u_{r}\right)dr+\int_{0}^{t}S_{t-r}\left(G_{r}u_{r}+g_{r}\right)d\eta_{r}+\int_{0}^{t}S_{t-r}h\left(r, u_{r}\right)dW_{r}, \quad t \in \left[0, T\right].
\end{equation}
To solve this equation, we need to establish the Young convolution integral 
\begin{equation} \label{ycidi}
	\int_{0}^{\cdot}S_{\cdot-r}Y_{r}d\eta_{r}
\end{equation}
for a stochastic process $Y: \left[0, T\right] \times \Omega \rightarrow \mathcal{H}_{\gamma}^{e}$
under certain conditions, where $\left(\mathcal{H}_{\gamma}\right)_{\gamma \in \mathbb{R}}$ is interpolation spaces 
corresponding to $L$. 
To this end, we give two different approaches. 
One approach is to regard $Y$ as a path from $\left[0, T\right]$ to 
$L^{m}\left(\Omega, \mathcal{H}_{\gamma}^{e}\right)$ for some $m \in \left[2, \infty\right)$. 
Such an approach is similar to defining the rough stochastic integration as in \cites{FHL21, LT23}. 
The other approach is to define \eqref{ycidi} pathwisely for a.s $\omega \in \Omega$, 
where $Y$ is required to have a beter time regularity. 
By a (continuous) mild solution of \eqref{yspdei}, 
we mean a process $u$ satisfying the equation \eqref{mf} where the Young convolution integral 
\begin{equation*}
	\int_{0}^{\cdot}S_{\cdot-r}\left(G_{r}u_{r}+g_{r}\right)d\eta_{r}
\end{equation*}
is defined through the first (second) approach. 
Then by a fixed-point argument together with some estimates, 
we get the existence and uniqueness of the (continuous) mild 
solution under suitable conditions.
The continuity of the solution map and spatial regularity of the mild solution 
are also obtained. 
To our best knowledge, this is the first study to SPDEs with Young drifts.\\
\indent
The paper is organized as follows. 
Section \ref{Sec2} contains preliminary notations and results. 
In Section \ref{Sec3}, we establish through two different approaches the Young convolution integrals 
for stochastic integrands. 
Mild solutions and continuous mild solutions are studied in Section \ref{Sec4} and \ref{Sec5}, respectively. 
In Section \ref{Sec6}, we provide a concrete example to illustrate our results. 

\section{Preliminaries}
\label{Sec2}
Throughout the paper, we fix $\alpha \in \left(\frac{1}{2}, 1\right)$ and let 
$\beta \in \left(0, 1\right)$ and $m \in \left[2, \infty\right]$. 
Write $a \lesssim b$ provided there exists a generic positive constant $C$ such that $a \leq Cb$. 
Fixing a finite time horizon $T > 0$, 
let $\Delta_{2}:=\left\{\left(s, t\right): 0 \leq s \leq t \leq T\right\}$ 
and $\Delta_{3}:=\left\{\left(s, r, t\right): 0 \leq s \leq r \leq t \leq T\right\}$. 
For $\left(s, t\right) \in \Delta_{2}$, denote by $\mathcal{P}\left[s, t\right]$ 
the set of all partitions of the interval $\left[s, t\right]$ and $\left|\pi\right|$ 
the mesh size of a partition $\pi \in \mathcal{P}\left[s, t\right]$. 
For Banach spaces $V$ and $\bar{V}$, define $\mathcal{L}\left(V, \bar{V}\right)$ as 
the space of bounded linear operators from $V$ to $\bar{V}$, endowed with the operator norm. 
Define $C\left(V, \bar{V}\right)$ as the space of bounded continuous maps 
from $V$ to $\bar{V}$, endowed with the maximum-norm.\\
\indent
Let $\left(\Omega, \mathcal{F}, \mathcal{F}_{t}, \mathbb{P}\right)$ 
be a filtered probability space satisfying the usual conditions and 
carrying a $d$-dimensional Brownian motion $W$. 
For a Banach space $\left(V, \left|\cdot\right|\right)$ and 
subfield $\mathcal{G}$ of $\mathcal{F}$, 
define $L^{m}\left(\Omega, \mathcal{G}, V\right)$ as 
the space of $\mathcal{G}$-measurable $L^{m}$-integrable 
$V$-valued random variables $\xi$, endowed with the norm 
$\left\|\xi\right\|_{m, V}:=\left\|\left|\xi\right|_{V}\right\|_{m}$. 
For simplicity write $L^{m}\left(\Omega, V\right):=L^{m}\left(\Omega, \mathcal{F}, V\right)$. 
\indent
\subsection{Analytic semigroups and interpolation spaces}
As in \cite{GH19}, let $\left(\mathcal{H}, \left|\cdot\right|\right)$ be a Hilbert space and 
$L: D\left(L\right) \subset \mathcal{H} \rightarrow \mathcal{H}$ 
be a linear operator generating an analytic semigroup $\left(S_{t}\right)_{t \geq 0}$. 
Assume without loss of generality that there exists a positive constant $\nu$ such that 
\begin{equation*}
	\left|S_{t}u\right| \lesssim e^{-\nu t}\left|u\right|, \quad \forall t \in \left[0, T\right], \quad \forall u \in \mathcal{H}.
\end{equation*}
By standard analytic semigroup theory (see \cites{P83, H09}), for $\gamma > 0$, 
we can define the bounded injective operator 
\begin{equation*}
	\left(-L\right)^{-\gamma}:=\frac{1}{\Gamma\left(\gamma\right)}\int_{0}^{\infty}r^{\gamma-1}S_{r}dr,
\end{equation*}
and $\left(-L\right)^{\gamma}$ as the inverse of $\left(-L\right)^{-\gamma}$. 
Then define $\mathcal{H}_{\gamma}:=D\left(\left(-L\right)^{\gamma}\right)$ endowed with the norm 
$\left|u\right|_{\gamma}:=\left|\left(-L\right)^{\gamma}u\right|$ 
and $\mathcal{H}_{-\gamma}$ as the completion of $\mathcal{H}$ for the norm 
$\left|u\right|_{-\gamma}:=\left|\left(-L\right)^{-\gamma}u\right|$. 
Write $\left(\mathcal{H}_{0}, \left|\cdot\right|_{0}\right):=\left(\mathcal{H}, \left|\cdot\right|\right)$. 
Then for every $\gamma \in \mathbb{R}$, $\mathcal{H}_{\gamma}$ is a Hilbert space and 
$\left(S_{t}\right)_{t \geq 0}$ is also an analytic semigroup on $\mathcal{H}_{\gamma}$. 
For $\gamma_{1} \leq \gamma_{2}$,  
$\mathcal{H}_{\gamma_{2}}$ is continuously embedded into $\mathcal{H}_{\gamma_{1}}$. 
The following results can be found in \cite{H09}*{Propositions 4.40 and 4.44}. 
\begin{proposition} \label{asp}
	~
	\begin{enumerate}[(i)]
		\item For $\gamma_{1} \leq \gamma_{2}$, we have 
		\begin{equation*}
			\left|S_{t}u\right|_{\gamma_{2}} \lesssim t^{\gamma_{1}-\gamma_{2}}\left|u\right|_{\gamma_{1}}, \quad \forall t \in \left(0, T\right], \quad \forall u \in \mathcal{H}_{\gamma_{1}}.
		\end{equation*}
		\item For $\gamma_{1} \leq \gamma_{2} < \gamma_{1}+1$, we have 
		\begin{equation*}
			\left|S_{t}u-u\right|_{\gamma_{1}} \lesssim t^{\gamma_{2}-\gamma_{1}}\left|u\right|_{\gamma_{2}}, \quad \forall t \in \left[0, T\right], \quad \forall u \in \mathcal{H}_{\gamma_{2}}.
		\end{equation*}
	\end{enumerate}
\end{proposition}
In the sequel of this paper, we let $\gamma, \gamma_{1}$ and $\gamma_{2}$ be any real numbers. 
For simplicity write $\mathcal{H}_{\gamma}^{e}:=\mathcal{L}\left(\mathbb{R}^{e}, \mathcal{H}_{\gamma}\right)$ 
and its norm is also denoted by $\left|\cdot\right|_{\gamma}$. 
For $f \in \mathcal{L}\left(\mathcal{H}_{\gamma_{1}}^{e_{1}}, \mathcal{H}_{\gamma_{2}}^{e_{2}}\right)$, 
define 
\begin{equation*}
	\left|f\right|_{\left(\gamma_{1}, \gamma_{2}\right)\text{-}op}:=\sup_{\left|u\right|_{\gamma_{1}} \leq 1}\left|fu\right|_{\gamma_{2}}.
\end{equation*}
Write 
$\left\|\cdot\right\|_{m, \gamma}:=\left\|\cdot\right\|_{m, \mathcal{H}_{\gamma}^{e}}$ 
and $\left\|\cdot\right\|_{m, \left(\gamma_{1}, \gamma_{2}\right)\text{-}op}:=\left\|\cdot\right\|_{m, \mathcal{L}\left(\mathcal{H}_{\gamma_{1}}^{e_{1}}, \mathcal{H}_{\gamma_{2}}^{e_{2}}\right)}$.
\subsection{Increment operators and Hölder type spaces}
For $Y: \left[0, T\right] \times \Omega \rightarrow \mathcal{H}_{\gamma}^{e}$, 
define $\delta Y, \hat{\delta}Y: \Delta_{2} \times \Omega \rightarrow \mathcal{H}_{\gamma}^{e}$ 
as the increment and mild increment of $Y$ respectively, i.e. 
\begin{equation*}
	\delta Y_{s, t}:=Y_{t}-Y_{s}, \quad \hat{\delta} Y_{s, t}:=Y_{t}-S_{t-s}Y_{s}, \quad \forall \left(s, t\right) \in \Delta_{2}.
\end{equation*}
Similarly, for $A: \Delta_{2} \times \Omega \rightarrow \mathcal{H}_{\gamma}^{e}$, 
define $\delta A, \hat{\delta}A: \Delta_{3} \times \Omega \rightarrow \mathcal{H}_{\gamma}^{e}$ by 
\begin{equation*}
	\delta A_{s, r, t}:=A_{s, t}-A_{s, r}-A_{r, t}, \quad \hat{\delta} A_{s, r, t}:=A_{s, t}-S_{t-r}A_{s, r}-A_{r, t}, \quad \forall \left(s, r, t\right) \in \Delta_{3}.
\end{equation*}
We say that $A$ is adapted 
if $A_{s, t}$ is $\mathcal{F}_{t}$-measurable for every $\left(s, t\right) \in \Delta_{2}$. 
Define 
$C_{2}^{\beta}L_{m}\mathcal{H}_{\gamma}^{e}$ as the space of measurable adapted processes 
$A: \Delta_{2} \times \Omega \rightarrow \mathcal{H}_{\gamma}^{e}$ such that 
$A \in C\left(\Delta_{2}, L^{m}\left(\Omega, \mathcal{H}_{\gamma}^{e}\right)\right)$ 
and 
\begin{equation*}
	\left\|A\right\|_{\beta, m, \gamma}:=\sup_{0 \leq s < t \leq T}\frac{\left\|A_{s, t}\right\|_{m, \gamma}}{|t-s|^{\beta}} < \infty. 
\end{equation*}
Note that $A \in C_{2}^{\beta}L_{m}\mathcal{H}_{\gamma}^{e}$ implies $A_{t, t}=0$ 
for every $t \in \left[0, T\right]$. 
For simplicity, write $CL_{m}\mathcal{H}_{\gamma}^{e}:=C\left(\left[0, T\right], L^{m}\left(\Omega, \mathcal{H}_{\gamma}^{e}\right)\right)$ 
and 
\begin{equation*}
	\left\|Y\right\|_{0, m, \gamma}:=\sup_{t \in \left[0, T\right]}\left\|Y_{t}\right\|_{m, \gamma}. 
\end{equation*}
Define $C^{\beta}L_{m}\mathcal{H}_{\gamma}^{e}$ (resp. $\hat{C}^{\beta}L_{m}\mathcal{H}_{\gamma}^{e}$) 
as the space of measurable adapted processes $Y: \left[0, T\right] \times \Omega \rightarrow \mathcal{H}_{\gamma}^{e}$ 
such that $Y \in CL_{m}\mathcal{H}_{\gamma}^{e}$ and 
$\left\|\delta Y\right\|_{\beta, m, \gamma}$ (resp. $\left\|\hat{\delta} Y\right\|_{\beta, m, \gamma}$) is finite. 
Then define $E^{\beta}L_{m}\mathcal{H}_{\gamma}^{e}:=C^{\beta}L_{m}\mathcal{H}_{\gamma-\beta}^{e} \cap CL_{m}\mathcal{H}_{\gamma}^{e}$, 
endowed with the norm 
\begin{equation*}
	\left\|Y\right\|_{E^{\beta}L_{m}\mathcal{H}_{\gamma}}:=\left\|\delta Y\right\|_{\beta, m, \gamma-\beta}+\left\|Y\right\|_{0, m, \gamma}.
\end{equation*}
To indicate the underlying time inteval $\left[0, T\right]$, 
we use notations $E^{\beta}L_{m}\mathcal{H}_{\gamma}^{e}\left[0, T\right]$ and 
$\left\|\cdot\right\|_{E^{\beta}L_{m}\mathcal{H}_{\gamma}\left[0, T\right]}$. 
\begin{proposition} \label{ne}
	$E^{\beta}L_{m}\mathcal{H}_{\gamma}^{e}=\hat{C}^{\beta}L_{m}\mathcal{H}_{\gamma-\beta}^{e} \cap CL_{m}\mathcal{H}_{\gamma}^{e}$ 
	and we have 
	\begin{equation} \label{ne1}
		\left\|Y\right\|_{E^{\beta}L_{m}\mathcal{H}_{\gamma}} \lesssim \left\|\hat{\delta} Y\right\|_{\beta, m, \gamma-\beta}+\left\|Y\right\|_{0, m, \gamma} \lesssim \left\|Y\right\|_{E^{\beta}L_{m}\mathcal{H}_{\gamma}}. 
	\end{equation}
\end{proposition}
\begin{proof}
	For every $Y \in CL_{m}\mathcal{H}_{\beta}^{e}$, by Proposition \ref{asp} 
	we have 
	\begin{equation*}
		\left\|\delta Y_{s, t}-\hat{\delta} Y_{s, t}\right\|_{m, \gamma-\beta}=\left\|S_{t-s}Y_{s}-Y_{s}\right\|_{m, \gamma-\beta} \lesssim \left|t-s\right|^{\beta}\left\|Y_{s}\right\|_{m, \gamma}, \quad \forall \left(s, t\right) \in \Delta_{2},
	\end{equation*}
	which gives 
	\begin{equation*}
		\left\|\delta Y-\hat{\delta} Y\right\|_{\beta, m, \gamma-\beta} \lesssim \left\|Y\right\|_{0, m, \gamma}. 
	\end{equation*}
	Hence, $E^{\beta}L_{m}\mathcal{H}_{\gamma}^{e}=\hat{C}^{\beta}L_{m}\mathcal{H}_{\gamma-\beta}^{e} \cap CL_{m}\mathcal{H}_{\gamma}^{e}$ and 
	the estimate \eqref{ne1} holds. 
\end{proof}
Similarly, define 
$E^{\beta}L_{\infty}\mathcal{L}_{\gamma_{1}, \gamma_{2}}\left(\mathcal{H}^{e_{1}}, \mathcal{H}^{e_{2}}\right)$ 
as the space of measurable adapted processes 
$f: \left[0, T\right] \times \Omega \rightarrow \mathcal{L}\left(\mathcal{H}_{\gamma_{1}-\beta}^{e_{1}}, \mathcal{H}_{\gamma_{2}-\beta}^{e_{2}}\right) \cap \mathcal{L}\left(\mathcal{H}_{\gamma_{1}}^{e_{1}}, \mathcal{H}_{\gamma_{2}}^{e_{2}}\right)$ 
such that 
\begin{align*}
	\left\|f\right\|_{E^{\beta}L_{\infty}\mathcal{L}_{\gamma_{1}, \gamma_{2}}}&:=\sup_{t \in \left[0, T\right]}\left(\left\|f_{t}\right\|_{\infty, \left(\gamma_{1}-\beta, \gamma_{2}-\beta\right)\text{-}op}+\left\|f_{t}\right\|_{\infty, \left(\gamma_{1}, \gamma_{2}\right)\text{-}op}\right)\\
	&\quad+\sup_{0 \leq s < t \leq T}\frac{\left\|\delta f_{s, t}\right\|_{\infty, \left(\gamma_{1}-\beta, \gamma_{2}-\beta\right)\text{-}op}}{|t-s|^{\beta}} < \infty. 
\end{align*}
\subsection{Pathwise Hölder continuous spaces}
For $\eta: \left[0, T\right] \rightarrow \mathbb{R}^{e}$, 
we can similarly define 
$\delta \eta$ as the increment of $\eta$ and 
\begin{equation*}
	\left|\delta \eta\right|_{\alpha}:=\sup_{0 \leq s < t \leq T}\frac{\left|\delta \eta_{s, t}\right|}{|t-s|^{\alpha}}. 
\end{equation*}
Then space of continuous paths $\eta: \left[0, T\right] \rightarrow \mathbb{R}^{e}$ such that 
$\left|\delta \eta\right|_{\alpha} < \infty$ is denoted by 
$C^{\alpha}\left(\left[0, T\right], \mathbb{R}^{e}\right)$. 
For $A: \Delta_{2} \rightarrow \mathcal{H}_{\gamma}^{e}$, 
define 
\begin{equation*}
	\left|A\right|_{\beta, \gamma}:=\sup_{0 \leq s < t \leq T}\frac{\left|A_{s, t}\right|_{\gamma}}{|t-s|^{\beta}}.
\end{equation*}
Denote by $L_{m}C_{2}^{\beta}\mathcal{H}_{\gamma}^{e}$ the space of measurable adapted processes 
$A: \Delta_{2} \times \Omega \rightarrow \mathcal{H}_{\gamma}^{e}$ such that 
$A \in L^{m}\left(\Omega, C\left(\Delta_{2}, \mathcal{H}_{\gamma}^{e}\right)\right)$ 
and 
\begin{equation*}
	\left\|A\right\|_{m, \beta, \gamma}:=\left\|\left|A\right|_{\beta, \gamma}\right\|_{m} < \infty. 
\end{equation*}
Then the Dominated Convergence Theorem gives $L_{m}C_{2}^{\beta}\mathcal{H}_{\gamma}^{e} \subset C_{2}^{\beta}L_{m}\mathcal{H}_{\gamma}^{e}$.  
Write $L_{m}C\mathcal{H}_{\gamma}^{e}:=L^{m}\left(\Omega, C\left(\left[0, T\right], \mathcal{H}_{\gamma}^{e}\right)\right)$ 
and 
\begin{equation*}
	\left\|Y\right\|_{m, 0, \gamma}=\left\|\left|Y\right|_{0, \gamma}\right\|_{m}:=\left\|\sup_{t \in \left[0, T\right]}\left|Y_{t}\right|_{\gamma}\right\|_{m}. 
\end{equation*}
Define $L_{m}C^{\beta}\mathcal{H}_{\gamma}^{e}$ (resp. $L_{m}\hat{C}^{\beta}\mathcal{H}_{\gamma}^{e}$) 
as the space of continuous adapted processes $Y: \left[0, T\right] \times \Omega \rightarrow \mathcal{H}_{\gamma}^{e}$ 
such that $Y \in L_{m}C\mathcal{H}_{\gamma}^{e}$ and 
$\left\|\delta Y\right\|_{m, \beta, \gamma}$ (resp. $\left\|\hat{\delta} Y\right\|_{m, \beta, \gamma}$) is finite. 
Similarly, $L_{m}C^{\beta}\mathcal{H}_{\gamma}^{e} \subset C^{\beta}L_{m}\mathcal{H}_{\gamma}^{e}$ and 
$L_{m}\hat{C}^{\beta}\mathcal{H}_{\gamma}^{e} \subset \hat{C}^{\beta}L_{m}\mathcal{H}_{\gamma}^{e}$. 
Conversely, we need the following version of the Kolmogorov Criterion. 
\begin{proposition} \label{kc}
	Let $Y \in C^{\beta}L_{m}\mathcal{H}_{\gamma}^{e}$ 
	(resp. $Y \in \hat{C}^{\beta}L_{m}\mathcal{H}_{\gamma}^{e}$) 
	be a continuous process such that $\frac{1}{\beta} < m < \infty$. 
	Then $Y \in L_{m}C^{\theta}\mathcal{H}_{\gamma}^{e}$ 
	(resp. $Y \in L_{m}\hat{C}^{\theta}\mathcal{H}_{\gamma}^{e}$) 
	for every $\theta \in \left(0, \beta-\frac{1}{m}\right)$ and we have 
	\begin{equation*}
		\left\|\delta Y\right\|_{m, \theta, \gamma} \lesssim T^{\varepsilon}\left\|\delta Y\right\|_{\beta, m, \gamma}\quad (resp.\ \left\|\hat{\delta} Y\right\|_{m, \theta, \gamma} \lesssim T^{\varepsilon}\left\|\hat{\delta}Y\right\|_{\beta, m, \gamma}),
	\end{equation*}
	for every $\varepsilon \in \left(0, \beta-\frac{1}{m}-\theta\right)$.
\end{proposition}
Define $L_{m}E^{\beta}\mathcal{H}_{\gamma}^{e}:=L_{m}C^{\beta}\mathcal{H}_{\gamma-\beta}^{e} \cap L_{m}C\mathcal{H}_{\gamma}^{e}$, 
endowed with the norm 
\begin{equation*}
	\left\|Y\right\|_{L_{m}E^{\beta}\mathcal{H}_{\gamma}}:=\left\|\delta Y\right\|_{m, \beta, \gamma-\beta}+\left\|Y\right\|_{m, 0, \gamma}.
\end{equation*}
Then $L_{m}E^{\beta}\mathcal{H}_{\gamma}^{e}$ is continuously embedded into 
$E^{\beta}L_{m}\mathcal{H}_{\gamma}^{e}$. 
Similar to Proposition \ref{ne}, we have the following result. 
\begin{proposition} \label{nep}
	$L_{m}E^{\beta}\mathcal{H}_{\gamma}^{e}=L_{m}\hat{C}^{\beta}\mathcal{H}_{\gamma-\beta}^{e} \cap L_{m}C\mathcal{H}_{\gamma}^{e}$ 
	and we have 
	\begin{equation*}
		\left\|Y\right\|_{L_{m}E^{\beta}\mathcal{H}_{\gamma}} \lesssim \left\|\hat{\delta} Y\right\|_{m, \beta, \gamma-\beta}+\left\|Y\right\|_{m, 0, \gamma} \lesssim \left\|Y\right\|_{L_{m}E^{\beta}\mathcal{H}_{\gamma}}. 
	\end{equation*}
\end{proposition}
Similarly, define 
$L_{\infty}E^{\beta}\mathcal{L}_{\gamma_{1}, \gamma_{2}}\left(\mathcal{H}^{e_{1}}, \mathcal{H}^{e_{2}}\right)$ 
as the space of continuous adapted processes 
$f: \left[0, T\right] \times \Omega \rightarrow \mathcal{L}\left(\mathcal{H}_{\gamma_{1}-\beta}^{e_{1}}, \mathcal{H}_{\gamma_{2}-\beta}^{e_{2}}\right) \cap \mathcal{L}\left(\mathcal{H}_{\gamma_{1}}^{e_{1}}, \mathcal{H}_{\gamma_{2}}^{e_{2}}\right)$ 
such that 
\begin{equation*}
	\left\|f\right\|_{L_{\infty}E^{\beta}\mathcal{L}_{\gamma_{1}, \gamma_{2}}}:=\left\|\sup_{t \in \left[0, T\right]}\left(\left|f_{t}\right|_{\left(\gamma_{1}-\beta, \gamma_{2}-\beta\right)\text{-}op}+\left|f_{t}\right|_{\left(\gamma_{1}, \gamma_{2}\right)\text{-}op}\right)+\sup_{0 \leq s < t \leq T}\frac{\left|\delta f_{s, t}\right|_{\left(\gamma_{1}-\beta, \gamma_{2}-\beta\right)\text{-}op}}{|t-s|^{\beta}}\right\|_{\infty} < \infty. 
\end{equation*}
Then $L_{\infty}E^{\beta}\mathcal{L}_{\gamma_{1}, \gamma_{2}}\left(\mathcal{H}^{e_{1}}, \mathcal{H}^{e_{2}}\right)$ is continuously embedded into 
$E^{\beta}L_{\infty}\mathcal{L}_{\gamma_{1}, \gamma_{2}}\left(\mathcal{H}^{e_{1}}, \mathcal{H}^{e_{2}}\right)$. 

\section{Young convolution integrals for stochastic integrands}
\label{Sec3}
In this section, we will establish through two different approaches the 
Young convolution integrals of stochastic processes 
against given $\eta \in C^{\alpha}\left(\left[0, T\right], \mathbb{R}^{e}\right)$. 
\subsection{Stochastic Young convolution integrals}
We first state the following version of the mild sewing lemma 
introduced by Gubinelli and Tindel \cite{GT10}*{Theorem 3.5}. 
It can be proved in an analogous manner to the proof of \cite{GH19}*{Theorem 2.4}. 
\begin{lemma} \label{msl}
	Let $A \in C_{2}^{\alpha}L_{m}\mathcal{H}_{\gamma}$. 
	Assume there exist positive constants $K, \varepsilon$ and a process 
	$\Lambda: \Delta_{3} \times \Omega \rightarrow \mathcal{H}_{\gamma}$ such that 
	\begin{equation*}
		\hat{\delta} A_{s, r, t}=S_{t-r}\Lambda_{s, r, t}, \quad \left\|\Lambda_{s, r, t}\right\|_{m, \gamma} \leq K\left|t-s\right|\left|t-r\right|^{\varepsilon}, \quad \forall \left(s, r, t\right) \in \Delta_{3}. 
	\end{equation*}
	Then there exists unique $\mathcal{A} \in \hat{C}^{\alpha}L_{m}\mathcal{H}_{\gamma}$ with $\mathcal{A}_{0}=0$ such that 
	\begin{equation*}
		\lim_{\pi \in \mathcal{P}\left[s, t\right], \left|\pi\right| \rightarrow 0}\left\|\hat{\delta}\mathcal{A}_{s, t}-\sum_{\left[r, v\right] \in \pi}S_{t-v}A_{r, v}\right\|_{m, \gamma}=0, \quad \forall \left(s, t\right) \in \Delta_{2}.
	\end{equation*}
	Moreover, for every $\theta \in \left[0, 1+\varepsilon\right)$ we have 
	\begin{equation*}
		\left\|\hat{\delta}\mathcal{A}_{s, t}-A_{s, t}\right\|_{m, \gamma+\theta} \lesssim K\left|t-s\right|^{1+\varepsilon-\theta}, \quad \forall \left(s, t\right) \in \Delta_{2}.
	\end{equation*}
\end{lemma}
Then we give the following result on stochastic Young convolution integrals. 
\begin{proposition} \label{yci}
	Let $\eta \in C^{\alpha}\left(\left[0, T\right], \mathbb{R}^{e}\right)$   
	and $Y \in E^{\beta}L_{m}\mathcal{H}_{\gamma}^{e}$ for some $\beta \in \left(1-\alpha, \alpha\right)$. 
	Then there exists unique 
	\begin{equation*}
		Z:=\int_{0}^{\cdot}S_{\cdot-r}Y_{r}d\eta_{r} \in \hat{C}^{\alpha}L_{m}\mathcal{H}_{\gamma-\beta} 
	\end{equation*}
	with $Z_{0}=0$ such that for every $\left(s, t\right) \in \Delta_{2}$, 
	\begin{equation} \label{ycid}
		\hat{\delta}Z_{s, t}=\int_{s}^{t}S_{t-r}Y_{r}d\eta_{r}=\lim_{\pi \in \mathcal{P}\left[s, t\right], \left|\pi\right| \rightarrow 0}\sum_{\left[r, v\right] \in \pi}S_{t-r}Y_{r}\delta \eta_{r, v}
	\end{equation}
	holds in $L^{m}\left(\Omega, \mathcal{H}_{\gamma-\beta}\right)$. 
	Moreover, $Z \in \hat{C}^{\alpha-\theta}L_{m}\mathcal{H}_{\gamma+\theta}$ for every $\theta \in \left[0, \alpha\right)$ and we have 
	\begin{equation} \label{ycie1}
		\left\|\hat{\delta}Z\right\|_{\alpha-\theta, m, \gamma+\theta} \lesssim \left\|Y\right\|_{E^{\beta}L_{m}\mathcal{H}_{\gamma}}\left|\delta \eta\right|_{\alpha}.
	\end{equation}
	As a consequence, the convolution map 
	\begin{equation*}
		C^{\alpha}\left(\left[0, T\right], \mathbb{R}^{e}\right) \times E^{\beta}L_{m}\mathcal{H}_{\gamma}^{e} \rightarrow E^{\beta}L_{m}\mathcal{H}_{\gamma+\theta}: \quad \left(\eta, Y\right) \mapsto Z
	\end{equation*}
	is a bounded bilinear map and we have 
	\begin{equation} \label{ycie2}
		\left\|Z\right\|_{E^{\beta}L_{m}\mathcal{H}_{\gamma+\theta}} \lesssim T^{\alpha-\left(\beta \vee \theta\right)}\left\|Y\right\|_{E^{\beta}L_{m}\mathcal{H}_{\gamma}}\left|\delta \eta\right|_{\alpha}.
	\end{equation}
\end{proposition}
\begin{proof}
	For every $\left(s, t\right) \in \Delta$, define $A_{s, t}:=S_{t-s}Y_{s}\delta \eta_{s, t}$. 
	Clearly, $A$ is a measurable adapted $L^{m}$-integrable process and 
	$A \in C\left(\Delta_{2}, L^{m}\left(\Omega, \mathcal{H}_{\gamma-\beta}\right)\right)$. 
	By Proposition \ref{asp}, 
	\begin{equation*}
		\left\|A_{s, t}\right\|_{m, \gamma-\beta} \lesssim \left\|Y_{s}\right\|_{m, \gamma-\beta}\left|\delta \eta_{s, t}\right| \lesssim \left\|Y\right\|_{E^{\beta}L_{m}\mathcal{H}_{\gamma}}\left|\delta \eta\right|_{\alpha}\left|t-s\right|^{\alpha}, \quad \forall \left(s, t\right) \in \Delta_{2},
	\end{equation*}
	which gives $A \in C_{2}^{\alpha}L_{m}\mathcal{H}_{\gamma-\beta}$. 
	For every $\left(s, r, t\right) \in \Delta_{3}$, we have 
	$\hat{\delta} A_{s, r, t}=-S_{t-r}\hat{\delta}Y_{s, r}\delta \eta_{r, t}$. 
	By Propositions \ref{asp} and \ref{ne}, 
	\begin{equation*}
		\left\|\hat{\delta}Y_{s, r}\delta \eta_{r, t}\right\|_{m, \gamma-\beta} \lesssim \left\|\hat{\delta}Y_{s, r}\right\|_{m, \gamma-\beta}\left|\delta \eta_{r, t}\right| \lesssim \left\|Y\right\|_{E^{\beta}L_{m}\mathcal{H}_{\gamma}}\left|\delta \eta\right|_{\alpha}\left|t-s\right|\left|t-r\right|^{\alpha+\beta-1}.
	\end{equation*}
	Then by Lemma \ref{msl}, there exists unique $Z \in \hat{C}^{\alpha}L_{m}\mathcal{H}_{\gamma-\beta}$ 
	with $Z_{0}=0$ such that \eqref{ycid} holds in $\mathcal{H}_{\gamma-\beta}$ 
	and for every $\theta \in \left[0, \alpha\right)$ we have 
	\begin{equation*}
		\left\|\hat{\delta}Z_{s, t}-S_{t-s}Y_{s}\delta \eta_{s, t}\right\|_{m, \gamma+\theta} \lesssim \left\|Y\right\|_{E^{\beta}L_{m}\mathcal{H}_{\gamma}}\left|\delta \eta\right|_{\alpha}\left|t-s\right|^{\alpha-\theta}, \quad \forall \left(s, t\right) \in \Delta_{2}.
	\end{equation*}
	By Proposition \ref{asp}, 
	\begin{equation*}
		\left\|S_{t-s}Y_{s}\delta \eta_{s, t}\right\|_{m, \gamma+\theta} \lesssim \left|t-s\right|^{-\theta}\left\|Y_{s}\right\|_{m, \gamma}\left|\delta \eta_{s, t}\right|\lesssim \left\|Y\right\|_{E^{\beta}L_{m}\mathcal{H}_{\gamma}}\left|\delta \eta\right|_{\alpha}\left|t-s\right|^{\alpha-\theta}, \quad \forall \left(s, t\right) \in \Delta_{2}.
	\end{equation*}
	Hence, $Z \in \hat{C}^{\alpha-\theta}L_{m}\mathcal{H}_{\gamma+\theta}$ and the estimate \eqref{ycie1} holds. 
	At last, since $Z_{0}=0$ we have 
	\begin{equation*}
		\left\|Z\right\|_{0, m, \gamma+\theta} \lesssim T^{\alpha-\theta}\left\|\hat{\delta}Z\right\|_{\alpha-\theta, m, \gamma+\theta} \lesssim T^{\alpha-\theta}\left\|Y\right\|_{E^{\beta}L_{m}\mathcal{H}_{\gamma}}\left|\delta \eta\right|_{\alpha}.
	\end{equation*}
	Note that 
	\begin{equation*}
		\left\|\hat{\delta}Z\right\|_{\beta, m, \gamma+\theta-\beta} \lesssim T^{\alpha-\left(\beta \vee \theta\right)}\left\|\hat{\delta}Z\right\|_{\alpha-\left(\theta-\beta\right)^{+}, m, \gamma+\left(\theta-\beta\right)^{+}} \lesssim T^{\alpha-\left(\beta \vee \theta\right)}\left\|Y\right\|_{E^{\beta}L_{m}\mathcal{H}_{\gamma}}\left|\delta \eta\right|_{\alpha}.
	\end{equation*}
	By Proposition \ref{ne}, $Z \in E^{\beta}L_{m}\mathcal{H}_{\gamma+\theta}$ and the estimate \eqref{ycie2} holds. 
\end{proof}
\subsection{Pathwise Young convolution integrals}
To define Young convolution integrals pathwisely, 
we need the following mild sewing lemma (see also \cite{GH19}*{Theorem 2.4}). 
\begin{lemma}
	Let $A \in L_{m}C_{2}^{\alpha}\mathcal{H}_{\gamma}$. 
	Assume there exist a positive constant $\varepsilon$, 
	$L^{m}$-integrable positive random variable $K$ and process 
	$\Lambda: \Delta_{3} \times \Omega \rightarrow \mathcal{H}_{\gamma}$ such that 
	\begin{equation*}
		\hat{\delta} A_{s, r, t}=S_{t-r}\Lambda_{s, r, t}, \quad \left|\Lambda_{s, r, t}\right|_{\gamma} \leq K\left|t-s\right|\left|t-r\right|^{\varepsilon}, \quad \forall \left(s, r, t\right) \in \Delta_{3}, \quad a.s.\ \omega \in \Omega.
	\end{equation*}
	Then there exists unique $\mathcal{A} \in L_{m}\hat{C}^{\alpha}\mathcal{H}_{\gamma}$ with $\mathcal{A}_{0}=0$ such that 
	\begin{equation*}
		\lim_{\pi \in \mathcal{P}\left[s, t\right], \left|\pi\right| \rightarrow 0}\left|\hat{\delta}\mathcal{A}_{s, t}-\sum_{\left[r, v\right] \in \pi}S_{t-v}A_{r, v}\right|_{\gamma}=0, \quad \forall \left(s, t\right) \in \Delta_{2}, \quad a.s.\ \omega \in \Omega.
	\end{equation*}
	Moreover, for every $\theta \in \left[0, 1+\varepsilon\right)$ we have 
	\begin{equation*}
		\left|\hat{\delta}\mathcal{A}_{s, t}-A_{s, t}\right|_{\gamma+\theta} \lesssim K\left|t-s\right|^{1+\varepsilon-\theta}, \quad \forall \left(s, t\right) \in \Delta_{2}, \quad a.s.\ \omega \in \Omega.
	\end{equation*}
\end{lemma}
Similar to Proposition \ref{yci}, we have the following result 
on pathwise Young convolution integrals. 
\begin{proposition} \label{ycip}
	Let $\eta \in C^{\alpha}\left(\left[0, T\right], \mathbb{R}^{e}\right)$   
	and $Y \in L_{m}E^{\beta}\mathcal{H}_{\gamma}^{e}$ for some $\beta \in \left(1-\alpha, \alpha\right)$. 
	Then there exists unique 
	\begin{equation*}
		Z:=\int_{0}^{\cdot}S_{\cdot-r}Y_{r}d\eta_{r} \in L_{m}\hat{C}^{\alpha}\mathcal{H}_{\gamma-\beta} 
	\end{equation*}
	with $Z_{0}=0$ such that for a.s $\omega \in \Omega$, 
	\begin{equation} \label{ycidp}
		\hat{\delta}Z_{s, t}=\int_{s}^{t}S_{t-r}Y_{r}d\eta_{r}=\lim_{\pi \in \mathcal{P}\left[s, t\right], \left|\pi\right| \rightarrow 0}\sum_{\left[r, v\right] \in \pi}S_{t-r}Y_{r}\delta \eta_{r, v}
	\end{equation}
	holds in $\mathcal{H}_{\gamma-\beta}$ for every $\left(s, t\right) \in \Delta_{2}$. 
	Moreover, $Z \in L_{m}\hat{C}^{\alpha-\theta}\mathcal{H}_{\gamma+\theta}$ for every $\theta \in \left[0, \alpha\right)$ and we have 
	\begin{equation*}
		\left\|\hat{\delta}Z\right\|_{m, \alpha-\theta, \gamma+\theta} \lesssim \left\|Y\right\|_{L_{m}E^{\beta}\mathcal{H}_{\gamma}}\left|\delta \eta\right|_{\alpha}.
	\end{equation*}
	As a consequence, the convolution map 
	\begin{equation*}
		C^{\alpha}\left(\left[0, T\right], \mathbb{R}^{e}\right) \times L_{m}E^{\beta}\mathcal{H}_{\gamma}^{e} \rightarrow L_{m}E^{\beta}\mathcal{H}_{\gamma+\theta}: \quad \left(\eta, Y\right) \mapsto Z
	\end{equation*}
	is a bounded bilinear map and we have 
	\begin{equation*}
		\left\|Z\right\|_{L_{m}E^{\beta}\mathcal{H}_{\gamma+\theta}} \lesssim T^{\alpha-\left(\beta \vee \theta\right)}\left\|Y\right\|_{L_{m}E^{\beta}\mathcal{H}_{\gamma}}\left|\delta \eta\right|_{\alpha}.
	\end{equation*}
\end{proposition}
\begin{remark}
	For $Y \in L_{m}E^{\beta}\mathcal{H}_{\gamma}^{e}$ with 
	$\beta \in \left(1-\alpha, \alpha\right)$, since 
	$L_{m}E^{\beta}\mathcal{H}_{\gamma}^{e} \subset E^{\beta}L_{m}\mathcal{H}_{\gamma}^{e}$ 
	we can define the Young convolution integral of $Y$ against $\eta$ by 
	either Proposition \ref{yci} or \ref{ycip}. 
	In view of \eqref{ycid} and \eqref{ycidp}, these two definitions are compatible and thus 
	we can use the same notation $\int_{0}^{\cdot}S_{\cdot-r}Y_{r}d\eta_{r}$. 
\end{remark}

\section{Mild solutions in $E^{\beta}L_{m}\mathcal{H}_{\gamma}$}
\label{Sec4}
In this and the next section, we fix $\gamma \in \mathbb{R}$, 
$\lambda \in \left[0, 1\right)$, $\mu \in \left[0, \frac{1}{2}\right)$ and 
$\nu \in \left[0, \alpha\right)$. 
Consider the following semilinear SPDE with a linear Young drift
\begin{equation} \label{yspde}
	\left\{
		\begin{aligned}
			&du_{t}=\left[Lu_{t}+f\left(t, u_{t}\right)\right]dt+\left(G_{t}u_{t}+g_{t}\right)d\eta_{t}+h\left(t, u_{t}\right)dW_{t}, \quad t \in \left(0, T\right],\\
			&u_{0}=\xi.
		\end{aligned}
	\right.	
\end{equation}
Here, $\eta \in C^{\alpha}\left(\left[0, T\right], \mathbb{R}^{e}\right)$, 
$\xi$ is an $\mathcal{H}_{\gamma}$-valued random variable, 
$f: \left[0, T\right] \times \Omega \times \mathcal{H}_{\gamma} \rightarrow \mathcal{H}_{\gamma-\lambda}$ 
and $h: \left[0, T\right] \times \Omega \times \mathcal{H}_{\gamma} \rightarrow \mathcal{H}_{\gamma-\mu}^{d}$ 
are progressively measurable vector fields, 
$G: \left[0, T\right] \times \Omega \rightarrow \mathcal{L}\left(\mathcal{H}_{\gamma}, \mathcal{H}_{\gamma-\nu}^{e}\right)$ 
and $g: \left[0, T\right] \times \Omega \rightarrow \mathcal{H}_{\gamma-\nu}^{e}$ 
are measurable adapted processes. 
Given $\beta \in \left(1-\alpha, \frac{1}{2}\right)$ 
and $m \in \left[2, \infty\right)$, 
we introduce the following definition of mild solutions in $E^{\beta}L_{m}\mathcal{H}_{\gamma}$. 
\begin{definition} \label{msd}
	We call $u \in E^{\beta}L_{m}\mathcal{H}_{\gamma}$ a mild solution of \eqref{yspde} if 
	$Gu+g \in E^{\beta}L_{m}\mathcal{H}_{\gamma-\nu}^{e}$ and 
	for every $t \in \left[0, T\right]$ and a.s. $\omega \in \Omega$, 
	\begin{equation*}
		u_{t}=S_{t}\xi+\int_{0}^{t}S_{t-r}f\left(r, u_{r}\right)dr+\int_{0}^{t}S_{t-r}\left(G_{r}u_{r}+g_{r}\right)d\eta_{r}+\int_{0}^{t}S_{t-r}h\left(r, u_{r}\right)dW_{r}
	\end{equation*}
	holds in $\mathcal{H}_{\gamma}$. 
\end{definition}
Then we introduce the following assumption. 
\begin{assumption} \label{as}
	~
	\begin{enumerate}[(i)]
		\item $\xi \in L^{m}\left(\Omega, \mathcal{F}_{0}, \mathcal{H}_{\gamma}\right)$; 
		\item $\left\|f\left(\cdot, 0\right)\right\|_{0, m, \gamma-\lambda} < \infty$ and 
		\begin{equation*}
			\left|f\left(t, u\right)-f\left(t, \bar{u}\right)\right|_{\gamma-\lambda} \lesssim \left|u-\bar{u}\right|_{\gamma}, \quad \forall t \in \left[0, T\right], \quad \forall u, \bar{u} \in \mathcal{H}_{\gamma};
		\end{equation*}
		\item $G \in E^{\beta}L_{\infty}\mathcal{L}_{\gamma, \gamma-\nu}\left(\mathcal{H}, \mathcal{H}^{e}\right)$ and $g \in E^{\beta}L_{m}\mathcal{H}_{\gamma-\nu}^{e}$; 
		\item $\left\|h\left(\cdot, 0\right)\right\|_{0, m, \gamma-\mu} < \infty$ and  
		\begin{equation*}
			\left|h\left(t, u\right)-h\left(t, \bar{u}\right)\right|_{\gamma-\mu} \lesssim \left|u-\bar{u}\right|_{\gamma}, \quad \forall t \in \left[0, T\right], \quad \forall u, \bar{u} \in \mathcal{H}_{\gamma}.
		\end{equation*}
	\end{enumerate}
\end{assumption}
\subsection{Existence and uniqueness}
We first give the following result on compositions. 
\begin{proposition} \label{lc}
	Let $u \in E^{\beta}L_{m}\mathcal{H}_{\gamma}$ 
	and $G \in E^{\beta}L_{\infty}\mathcal{L}_{\gamma, \gamma-\nu}\left(\mathcal{H}, \mathcal{H}^{e}\right)$. 
	Then $Gu \in E^{\beta}L_{m}\mathcal{H}_{\gamma-\nu}^{e}$ and we have 
	\begin{equation} \label{lce}
		\left\|Gu\right\|_{E^{\beta}L_{m}\mathcal{H}_{\gamma-\nu}} \lesssim \left\|G\right\|_{E^{\beta}L_{\infty}\mathcal{L}_{\gamma, \gamma-\nu}}\left\|u\right\|_{E^{\beta}L_{m}\mathcal{H}_{\gamma}}.
	\end{equation}
\end{proposition}
\begin{proof}
	Clearly, we have 
	\begin{equation*}
		\left\|G_{t}u_{t}\right\|_{m, \gamma-\nu} \lesssim \left\|G_{t}\right\|_{\infty, \left(\gamma, \gamma-\nu\right)\text{-}op}\left\|u_{t}\right\|_{m, \gamma} \lesssim \left\|G\right\|_{E^{\beta}L_{\infty}\mathcal{L}_{\gamma, \gamma-\nu}}\left\|u\right\|_{E^{\beta}L_{m}\mathcal{H}_{\gamma}}, \quad \forall t \in \left[0, T\right],
	\end{equation*}
	\begin{align*}
		\left\|G_{t}u_{t}-G_{s}u_{s}\right\|_{m, \gamma-\beta-\nu} &\leq \left\|\delta G_{s, t}u_{t}\right\|_{m, \gamma-\beta-\nu}+\left\|G_{s} \delta u_{s, t}\right\|_{m, \gamma-\beta-\nu}\\
		&\lesssim \left\|\delta G_{s, t}\right\|_{\infty, \left(\gamma-\beta, \gamma-\beta-\nu\right)\text{-}op}\left\|u_{t}\right\|_{m, \gamma-\beta}+\left\|G_{s}\right\|_{\infty, \left(\gamma-\beta, \gamma-\beta-\nu\right)\text{-}op}\left\|\delta u_{s, t}\right\|_{m, \gamma-\beta}\\
		&\lesssim \left\|G\right\|_{E^{\beta}L_{\infty}\mathcal{L}_{\gamma, \gamma-\nu}}\left\|u\right\|_{E^{\beta}L_{m}\mathcal{H}_{\gamma}}\left|t-s\right|^{\beta}, \quad \forall \left(s, t\right) \in \Delta_{2}.
	\end{align*}
	Hence, $Gu \in E^{\beta}L_{m}\mathcal{H}_{\gamma-\nu}^{e}$ and the estimate \eqref{lce} holds. 
\end{proof}
We now give the existence and uniqueness of the mild solution of \eqref{yspde}. 
\begin{theorem} \label{seu}
	Under Assumption \ref{as}, the Young SPDE \eqref{yspde} 
	has a unique mild solution $u \in E^{\beta}L_{m}\mathcal{H}_{\gamma}$ and we have 
	\begin{equation} \label{se}
		\left\|u\right\|_{E^{\beta}L_{m}\mathcal{H}_{\gamma}} \lesssim \left\|\xi\right\|_{m, \gamma}+\left\|f\left(\cdot, 0\right)\right\|_{0, m, \gamma-\lambda}+\left\|h\left(\cdot, 0\right)\right\|_{0, m, \gamma-\mu}+\left\|g\right\|_{E^{\beta}L_{m}\mathcal{H}_{\gamma-\nu}},
	\end{equation}
	for a hidden prefactor depending only on $T, \left|\delta \eta\right|_{\alpha}$ and $\left\|G\right\|_{E^{\beta}L_{\infty}\mathcal{L}_{\gamma, \gamma-\nu}}$. 
\end{theorem}
\begin{proof}
	Let $\varepsilon \in \left(0, 1\right]$ be a constant waiting to be determined. 
	We first show the existence and uniqueness for $T \leq \varepsilon$. 
	For any $u \in E^{\beta}L_{m}\mathcal{H}_{\gamma}$, by Proposition \ref{lc}, 
	$Y:=Gu+g \in E^{\beta}L_{m}\mathcal{H}_{\gamma-\nu}^{e}$. 
	Then by Proposition \ref{yci}, 
	$Z:=\int_{0}^{\cdot}S_{\cdot-r}Y_{r}d\eta_{r} \in E^{\beta}L_{m}\mathcal{H}_{\gamma}$ and 
	\begin{align*}
		\left\|Z\right\|_{E^{\beta}L_{m}\mathcal{H}_{\gamma}} &\lesssim T^{\alpha-\left(\beta \vee \nu\right)}\left\|Y\right\|_{E^{\beta}L_{m}\mathcal{H}_{\gamma-\nu}}\left|\delta \eta\right|_{\alpha}\\
		&\lesssim T^{\alpha-\left(\beta \vee \nu\right)}\left(\left\|G\right\|_{E^{\beta}L_{\infty}\mathcal{L}_{\gamma, \gamma-\nu}}\left\|u\right\|_{E^{\beta}L_{m}\mathcal{H}_{\gamma}}+\left\|g\right\|_{E^{\beta}L_{m}\mathcal{H}_{\gamma-\nu}}\right)\left|\delta \eta\right|_{\alpha}. 
	\end{align*} 
	Define 
	\begin{equation*}
		F:=\int_{0}^{\cdot}S_{\cdot-r}f\left(r, u_{r}\right)dr, \quad H:=\int_{0}^{\cdot}S_{\cdot-r}h\left(r, u_{r}\right)dW_{r}, \quad \Phi\left(u\right):=S\xi+F+Z+H.
	\end{equation*}
	Since $S$ is bounded and strongly continuous on $\mathcal{H}_{\gamma}$, 
	we have $S\xi \in L_{m}C\mathcal{H}_{\gamma}$ and 
	\begin{equation*}
		\left\|S\xi\right\|_{m, 0, \gamma} \lesssim \left\|\xi\right\|_{m, \gamma}. 
	\end{equation*}
	Combined with $\hat{\delta}S\xi=0$, by Proposition \ref{nep}, 
	$S\xi \in L_{m}E^{\beta}\mathcal{H}_{\gamma} \subset E^{\beta}L_{m}\mathcal{H}_{\gamma}$ and 
	\begin{equation} \label{Sxie}
		\left\|S\xi\right\|_{E^{\beta}L_{m}\mathcal{H}_{\gamma}} \lesssim \left\|S\xi\right\|_{L_{m}E^{\beta}\mathcal{H}_{\gamma}} \lesssim \left\|\xi\right\|_{m, \gamma}. 
	\end{equation}
	By Proposition \ref{asp}, applying Minkowski's Inequality, 
	\begin{align*}
		\left\|\hat{\delta}F_{s, t}\right\|_{m, \gamma-\beta} &\lesssim \left(\mathbb{E}\left(\int_{s}^{t}\left|S_{t-r}f\left(r, u_{r}\right)\right|_{\gamma-\beta}dr\right)^{m}\right)^{\frac{1}{m}} \lesssim \int_{s}^{t}\left|t-r\right|^{-\left(\lambda-\beta\right)^{+}}\left\|f\left(r, u_{r}\right)\right\|_{m, \gamma-\lambda}dr\\
		&\lesssim \left(\left\|u\right\|_{E^{\beta}L_{m}\mathcal{H}_{\gamma}}+\left\|f\left(\cdot, 0\right)\right\|_{0, m, \gamma-\lambda}\right)\left|t-s\right|^{1-\left(\lambda-\beta\right)^{+}}, \quad \forall \left(s, t\right) \in \Delta_{2},
	\end{align*}
	\begin{align*}
		\left\|\hat{\delta}F_{s, t}\right\|_{m, \gamma} &\lesssim \left(\mathbb{E}\left(\int_{s}^{t}\left|S_{t-r}f\left(r, u_{r}\right)\right|_{\gamma}dr\right)^{m}\right)^{\frac{1}{m}} \lesssim \int_{s}^{t}\left|t-r\right|^{-\lambda}\left\|f\left(r, u_{r}\right)\right\|_{m, \gamma-\lambda}dr\\
		&\lesssim \left(\left\|u\right\|_{E^{\beta}L_{m}\mathcal{H}_{\gamma}}+\left\|f\left(\cdot, 0\right)\right\|_{0, m, \gamma-\lambda}\right)\left|t-s\right|^{1-\lambda}, \quad \forall \left(s, t\right) \in \Delta_{2}.
	\end{align*}
	Then $F \in \hat{C}^{1-\left(\lambda-\beta\right)^{+}}L_{m}\mathcal{H}_{\gamma-\beta} \cap \hat{C}^{1-\lambda}L_{m}\mathcal{H}_{\gamma}$ and 
	\begin{equation} \label{Fe}
		\left\|\hat{\delta}F\right\|_{1-\left(\lambda-\beta\right)^{+}, m, \gamma-\beta}+\left\|\hat{\delta}F\right\|_{1-\lambda, m, \gamma} \lesssim \left\|u\right\|_{E^{\beta}L_{m}\mathcal{H}_{\gamma}}+\left\|f\left(\cdot, 0\right)\right\|_{0, m, \gamma-\lambda}.
	\end{equation}
	By Proposition \ref{ne}, $F \in E^{\beta}L_{m}\mathcal{H}_{\gamma}$ and 
	\begin{align*}
		\left\|F\right\|_{E^{\beta}L_{m}\mathcal{H}_{\gamma}} &\lesssim T^{1-\left(\beta \vee \lambda\right)}\left\|\hat{\delta}F\right\|_{1-\left(\lambda-\beta\right)^{+}, m, \gamma-\beta}+T^{1-\lambda}\left\|\hat{\delta}F\right\|_{1-\lambda, m, \gamma}\\
		&\lesssim T^{1-\left(\beta \vee \lambda\right)}\left(\left\|u\right\|_{E^{\beta}L_{m}\mathcal{H}_{\gamma}}+\left\|f\left(\cdot, 0\right)\right\|_{0, m, \gamma-\lambda}\right).
	\end{align*}
	Similarly, by Proposition \ref{asp}, applying Minkowski's Inequality and the BDG Inequality, 
	\begin{align*}
		\left\|\hat{\delta}H_{s, t}\right\|_{m, \gamma-\beta} &\lesssim \left(\mathbb{E}\left(\int_{s}^{t}\left|S_{t-r}h\left(r, u_{r}\right)\right|_{\gamma-\beta}^{2}dr\right)^{\frac{m}{2}}\right)^{\frac{1}{m}} \lesssim \left(\int_{s}^{t}\left|t-r\right|^{-2\left(\mu-\beta\right)^{+}}\left\|h\left(r, u_{r}\right)\right\|_{m, \gamma-\mu}^{2}dr\right)^{\frac{1}{2}}\\
		&\lesssim \left(\left\|u\right\|_{E^{\beta}L_{m}\mathcal{H}_{\gamma}}+\left\|h\left(\cdot, 0\right)\right\|_{0, m, \gamma-\mu}\right)\left|t-s\right|^{\frac{1}{2}-\left(\mu-\beta\right)^{+}}, \quad \forall \left(s, t\right) \in \Delta_{2},
	\end{align*}
	\begin{align*}
		\left\|\hat{\delta}H_{s, t}\right\|_{m, \gamma} &\lesssim \left(\mathbb{E}\left(\int_{s}^{t}\left|S_{t-r}h\left(r, u_{r}\right)\right|_{\gamma}^{2}dr\right)^{\frac{m}{2}}\right)^{\frac{1}{m}} \lesssim \left(\int_{s}^{t}\left|t-r\right|^{-2\mu}\left\|h\left(r, u_{r}\right)\right\|_{m, \gamma-\mu}^{2}dr\right)^{\frac{1}{2}}\\
		&\lesssim \left(\left\|u\right\|_{E^{\beta}L_{m}\mathcal{H}_{\gamma}}+\left\|h\left(\cdot, 0\right)\right\|_{0, m, \gamma-\mu}\right)\left|t-s\right|^{\frac{1}{2}-\mu}, \quad \forall \left(s, t\right) \in \Delta_{2}.
	\end{align*}
	Then $H \in \hat{C}^{\frac{1}{2}-\left(\mu-\beta\right)^{+}}L_{m}\mathcal{H}_{\gamma-\beta} \cap \hat{C}^{\frac{1}{2}-\mu}L_{m}\mathcal{H}_{\gamma}$ and 
	\begin{equation} \label{He}
		\left\|\hat{\delta}H\right\|_{\frac{1}{2}-\left(\mu-\beta\right)^{+}, m, \gamma-\beta}+\left\|\hat{\delta}H\right\|_{\frac{1}{2}-\mu, m, \gamma} \lesssim \left\|u\right\|_{E^{\beta}L_{m}\mathcal{H}_{\gamma}}+\left\|h\left(\cdot, 0\right)\right\|_{0, m, \gamma-\mu}.
	\end{equation}
	By Proposition \ref{ne}, $H \in E^{\beta}L_{m}\mathcal{H}_{\gamma}$ and 
	\begin{align*}
		\left\|H\right\|_{E^{\beta}L_{m}\mathcal{H}_{\gamma}} &\lesssim T^{\frac{1}{2}-\left(\beta \vee \mu\right)}\left\|\hat{\delta}H\right\|_{\frac{1}{2}-\left(\mu-\beta\right)^{+}, m, \gamma-\beta}+T^{\frac{1}{2}-\beta}\left\|\hat{\delta}H\right\|_{\frac{1}{2}-\mu, m, \gamma}\\
		&\lesssim T^{\frac{1}{2}-\left(\beta \vee \mu\right)}\left(\left\|u\right\|_{E^{\beta}L_{m}\mathcal{H}_{\gamma}}+\left\|h\left(\cdot, 0\right)\right\|_{0, m, \gamma-\mu}\right).
	\end{align*}
	Therefore, $\Phi\left(u\right) \in E^{\beta}L_{m}\mathcal{H}_{\gamma}$ and there exists $\sigma > 0$ depending only on $\beta, \lambda, \mu$ and $\nu$ such that 
	\begin{align}
		\left\|\Phi\left(u\right)\right\|_{E^{\beta}L_{m}\mathcal{H}_{\gamma}} &\lesssim T^{\sigma}\left(1+\left\|G\right\|_{E^{\beta}L_{\infty}\mathcal{L}_{\gamma, \gamma-\nu}}\left|\delta \eta\right|_{\alpha}\right)\left\|u\right\|_{E^{\beta}L_{m}\mathcal{H}_{\gamma}} \notag\\
		&\quad+\left\|\xi\right\|_{m, \gamma}+\left\|f\left(\cdot, 0\right)\right\|_{0, m, \gamma-\lambda}+\left\|h\left(\cdot, 0\right)\right\|_{0, m, \gamma-\mu}+\left\|g\right\|_{E^{\beta}L_{m}\mathcal{H}_{\gamma-\nu}}\left|\delta \eta\right|_{\alpha}. \label{phie}
	\end{align}
	For any other $\bar{u} \in E^{\beta}L_{m}\mathcal{H}_{\gamma}$, note that 
	\begin{align*}
		\Phi\left(u\right)-\Phi\left(\bar{u}\right)&=\int_{0}^{\cdot}S_{\cdot-r}\left[f\left(r, u_{r}\right)-f\left(r, \bar{u}_{r}\right)\right]dr+\int_{0}^{\cdot}S_{\cdot-r}G_{r}\left(u_{r}-\bar{u}_{r}\right)d\eta_{r}\\
		&\quad+\int_{0}^{\cdot}S_{\cdot-r}\left[h\left(r, u_{r}\right)-f\left(r, \bar{u}_{r}\right)\right]dW_{r}.
	\end{align*}
	Analogous to the above arguments, we have 
	\begin{equation*}
		\left\|\Phi\left(u\right)-\Phi\left(\bar{u}\right)\right\|_{E^{\beta}L_{m}\mathcal{H}_{\gamma}} \lesssim T^{\sigma}\left(1+\left\|G\right\|_{E^{\beta}L_{\infty}\mathcal{L}_{\gamma, \gamma-\nu}}\left|\delta \eta\right|_{\alpha}\right)\left\|u-\bar{u}\right\|_{E^{\beta}L_{m}\mathcal{H}_{\gamma}}. 
	\end{equation*}
	Hence, we can choose $\varepsilon \in \left(0, 1\right]$ such that 
	$\Phi$ is a $\frac{1}{2}$-contraction in $E^{\beta}L_{m}\mathcal{H}_{\gamma}$ for $T \leq \varepsilon$. 
	Applying the Banach Fixed-point Theorem, $\Phi$ has a unique fixed point $u$ in $E^{\beta}L_{m}\mathcal{H}_{\gamma}$, 
	which is the unique mild solution of \eqref{yspde}.\\
	\indent
	For arbitrary $T$, consider a partition $0=t_{0}<\cdots<t_{N}=T$ such that 
	$t_{i+1}-t_{i} \leq \varepsilon$ for $i=0, \cdots, N-1$. 
	Define $u_{0}=u_{t_{0}}^{0}:=\xi$ and then define $u^{i}$ recursively on 
	$\left(t_{i-1},t_{i}\right]$ for $i=1, \cdots, N$ by the mild solution 
	in $E^{\beta}L_{m}\mathcal{H}_{\gamma}\left[t_{i-1}, t_{i}\right]$ to the Young SPDE 
	\begin{equation*}
		\left\{
			\begin{aligned}
				&du_{t}^{i}=\left[Lu_{t}^{i}+f\left(t, u_{t}^{i}\right)\right]dt+\left(G_{t}u_{t}^{i}+g_{t}\right)d\eta_{t}+h\left(t, u_{t}^{i}\right)dW_{t}, \quad t \in \left(t_{i-1}, t_{i}\right],\\
				&u_{t_{i-1}}^{i}=u_{t_{i-1}}^{i-1}.
			\end{aligned}
		\right.	
	\end{equation*}
	Analogous to the above arguments but replacing $\xi$ by $u^{i-1}_{t_{i-1}}$, we can get 
	the existence and uniqueness of $u^{i}$, since $\varepsilon$ does not depend on $\xi$. 
	Define $u_{t}:=u_{t}^{t_{i}}$ for every 
	$t \in \left(t_{i-1},t_{i}\right]$ and $i=1, \cdots, N$. 
	Then $u \in E^{\beta}L_{m}\mathcal{H}_{\gamma}\left[0, T\right]$ is the unique mild solution of \eqref{yspde}.\\
	\indent
	At last, we show the estimate \eqref{se}. For $T \leq \varepsilon$, 
	the estimate \eqref{se} is implied by \eqref{phie}. 
	For arbitrary $T$, we similarly have 
	\begin{equation*}
		\left\|u^{i}\right\|_{E^{\beta}L_{m}\mathcal{H}_{\gamma}\left[t_{i-1}, t_{i}\right]} \lesssim \left\|u_{t_{i-1}}^{i-1}\right\|_{m, \gamma}+\left\|f\left(\cdot, 0\right)\right\|_{0, m, \gamma-\lambda}+\left\|h\left(\cdot, 0\right)\right\|_{0, m, \gamma-\mu}+\left\|g\right\|_{E^{\beta}L_{m}\mathcal{H}_{\gamma-\nu}}, 
	\end{equation*}
	for every $i=1, 2, \cdots, N$. By induction, we get the estimate \eqref{se}. 
\end{proof}
\subsection{Continuity of the mild solution map}
Next, we show the continuity of the mild solution map. 
\begin{theorem} \label{csm}
	Let Assumption \ref{as} hold and additionally 
	$\bar{\eta} \in C^{\alpha}\left(\left[0, T\right], \mathbb{R}^{e}\right)$, 
	$\bar{\xi} \in L^{m}\left(\Omega, \mathcal{F}_{0}, \mathcal{H}_{\gamma}\right)$. 
	Assume $\left|\delta \eta\right|_{\alpha}, \left|\delta \bar{\eta}\right|_{\alpha}, \left\|\xi\right\|_{m, \gamma}, \left\|\bar{\xi}\right\|_{m, \gamma} \leq R$ for some $R \geq 0$. 
	Let $u \in E^{\beta}L_{m}\mathcal{H}_{\gamma}$ be the mild solution of \eqref{yspde} and 
	$\bar{u} \in E^{\beta}L_{m}\mathcal{H}_{\gamma}$ be the mild solution of \eqref{yspde} with $\eta$ and $\xi$ 
	replaced by $\bar{\eta}$ and $\bar{\xi}$, respectively. 
	Then we have 
	\begin{equation} \label{smc}
		\left\|u-\bar{u}\right\|_{E^{\beta}L_{m}\mathcal{H}_{\gamma}} \lesssim \left|\delta \eta-\delta \bar{\eta}\right|_{\alpha}+\left\|\xi-\bar{\xi}\right\|_{m, \gamma}, 
	\end{equation}
	for a hidden prefactor depending only on $T, R, \left\|f\left(\cdot, 0\right)\right\|_{0, m, \gamma-\lambda}, \left\|h\left(\cdot, 0\right)\right\|_{0, m, \gamma-\mu}, \left\|G\right\|_{E^{\beta}L_{\infty}\mathcal{L}_{\gamma, \gamma-\nu}}$ and $\left\|g\right\|_{E^{\beta}L_{m}\mathcal{H}_{\gamma-\nu}}$. 
	As a consequence, the mild solution map 
	\begin{equation*}
		C^{\alpha}\left(\left[0, T\right], \mathbb{R}^{e}\right) \times L^{m}\left(\Omega, \mathcal{F}_{0}, \mathcal{H}_{\gamma}\right) \rightarrow E^{\beta}L_{m}\mathcal{H}_{\gamma}: \quad \left(\eta, \xi\right) \mapsto u
	\end{equation*}
	is locally Lipschitz continuous. 
\end{theorem}
\begin{proof}
	Recall the definition of $Y, Z, F$ and $H$ in the proof of Theorem \ref{seu}. 
	We similarly define $\bar{Y}, \bar{Z}, \bar{F}$ and $\bar{H}$. 
	By Theorem \ref{seu}, there exists $M \geq 0$ depending only on 
	$T, \left\|f\left(\cdot, 0\right)\right\|_{0, m, \gamma-\lambda}, \left\|h\left(\cdot, 0\right)\right\|_{0, m, \gamma-\mu}$, $\left\|G\right\|_{E^{\beta}L_{\infty}\mathcal{L}_{\gamma, \gamma-\nu}}, \left\|g\right\|_{E^{\beta}L_{m}\mathcal{H}_{\gamma-\nu}}$ and $R$, 
	such that $\left\|u\right\|_{E^{\beta}L_{m}\mathcal{H}_{\gamma}}, \left\|\bar{u}\right\|_{E^{\beta}L_{m}\mathcal{H}_{\gamma}} \leq M$. 
	Then by Propositions \ref{yci} and \ref{lc}, we have 
	\begin{align*}
		\left\|Z-\bar{Z}\right\|_{E^{\beta}L_{m}\mathcal{H}_{\gamma}} &\lesssim \left\|\int_{0}^{\cdot}S_{\cdot-r}\left(Y_{r}-\bar{Y}_{r}\right)d\eta_{r}\right\|_{E^{\beta}L_{m}\mathcal{H}_{\gamma}}+\left\|\int_{0}^{\cdot}S_{\cdot-r}\bar{Y}_{r}d\left(\eta_{r}-\bar{\eta}_{r}\right)\right\|_{E^{\beta}L_{m}\mathcal{H}_{\gamma}}\\
		&\lesssim T^{\alpha-\left(\beta \vee \nu\right)}\left\|Y-\bar{Y}\right\|_{E^{\beta}L_{m}\mathcal{H}_{\gamma-\nu}}\left|\delta \eta\right|_{\alpha}+T^{\alpha-\left(\beta \vee \nu\right)}\left\|\bar{Y}\right\|_{E^{\beta}L_{m}\mathcal{H}_{\gamma-\nu}}\left|\eta-\bar{\eta}\right|_{\alpha}\\
		&\lesssim T^{\alpha-\left(\beta \vee \nu\right)}\left\|u-\bar{u}\right\|_{E^{\beta}L_{m}\mathcal{H}_{\gamma}}+\left|\eta-\bar{\eta}\right|_{\alpha}. 
	\end{align*}
	Analogous to the proof of Theorem \ref{seu}, we have 
	\begin{equation*}
		\left\|S\left(\xi-\bar{\xi}\right)\right\|_{E^{\beta}L_{m}\mathcal{H}_{\gamma}} \lesssim \left\|\xi-\bar{\xi}\right\|_{m, \gamma},
	\end{equation*}
	\begin{equation*}
		\left\|F-\bar{F}\right\|_{E^{\beta}L_{m}\mathcal{H}_{\gamma}} \lesssim T^{1-\left(\beta \vee \lambda\right)}\left\|u-\bar{u}\right\|_{E^{\beta}L_{m}\mathcal{H}_{\gamma}}, \quad \left\|H-\bar{H}\right\|_{E^{\beta}L_{m}\mathcal{H}_{\gamma}} \lesssim T^{\frac{1}{2}-\left(\beta \vee \mu\right)}\left\|u-\bar{u}\right\|_{E^{\beta}L_{m}\mathcal{H}_{\gamma}}.
	\end{equation*}
	Then there exists $\sigma > 0$ depending only on $\beta, \lambda, \mu$ and $\nu$ such that 
	\begin{equation*}
		\left\|u-\bar{u}\right\|_{E^{\beta}L_{m}\mathcal{H}_{\gamma}} \lesssim T^{\sigma}\left\|u-\bar{u}\right\|_{E^{\beta}L_{m}\mathcal{H}_{\gamma}}+\left|\delta \eta-\delta \bar{\eta}\right|_{\alpha}+\left\|\xi-\bar{\xi}\right\|_{m, \gamma}. 
	\end{equation*}
	Hence, for $T$ sufficiently small we get the estimate \eqref{smc}. 
	The general result can be obtained by induction. 
\end{proof}
\subsection{Regularity of the mild solution}
Then we show that the mild solution of \eqref{yspde} has a better spatial regularity after some time. 
\begin{proposition} \label{sr}
	Let Assumption \ref{as} hold and $u \in E^{\beta}L_{m}\mathcal{H}_{\gamma}$ 
	be the mild solution of \eqref{yspde}. 
	Then $u \in \hat{C}^{\left(1-\lambda\right) \wedge \left(\frac{1}{2}-\mu\right) \wedge \left(\alpha-\nu\right)-\theta}L_{m}\mathcal{H}_{\gamma+\theta}\left[t, T\right]$ 
	for every $t \in \left(0, T\right]$ and $0 \leq \theta < \left(1-\lambda\right) \wedge \left(\frac{1}{2}-\mu\right) \wedge \left(\alpha-\nu\right)$ 
	and we have 
	\begin{equation} \label{sre}
		\left\|u_{t}\right\|_{m, \gamma+\theta} \lesssim t^{-\theta}\left\|\xi\right\|_{m, \gamma}+\left\|f\left(\cdot, 0\right)\right\|_{0, m, \gamma-\lambda}+\left\|h\left(\cdot, 0\right)\right\|_{0, m, \gamma-\mu}+\left\|g\right\|_{E^{\beta}L_{m}\mathcal{H}_{\gamma-\nu}}
	\end{equation}
	for a hidden prefactor depending only on $T, \left|\delta \eta\right|_{\alpha}$ and $\left\|G\right\|_{E^{\beta}L_{\infty}\mathcal{L}_{\gamma, \gamma-\nu}}$. 
\end{proposition}
\begin{proof}
	Recall the definition of $Y, Z, F$ and $H$ in the proof of Theorem \ref{seu}. 
	For every fixed $0 \leq \theta < \left(1-\lambda\right) \wedge \left(\frac{1}{2}-\mu\right) \wedge \left(\alpha-\nu\right)$, 
	by Propositions \ref{yci} and \ref{lc}, $Z \in \hat{C}^{\alpha-\nu-\theta}L_{m}\mathcal{H}_{\gamma+\theta}$ and we have 
	\begin{equation*}
		\left\|\hat{\delta}Z\right\|_{\alpha-\nu-\theta, m, \gamma+\theta} \lesssim \left\|Y\right\|_{E^{\beta}L_{m}\mathcal{H}_{\gamma-\nu}}\left|\delta \eta\right|_{\alpha} \lesssim \left(\left\|G\right\|_{E^{\beta}L_{\infty}\mathcal{L}_{\gamma, \gamma-\nu}}\left\|u\right\|_{E^{\beta}L_{m}\mathcal{H}_{\gamma}}+\left\|g\right\|_{E^{\beta}L_{m}\mathcal{H}_{\gamma-\nu}}\right)\left|\delta \eta\right|_{\alpha}. 
	\end{equation*}
	By Proposition \ref{asp}, we have 
	\begin{equation*}
		\left\|S_{t}\xi\right\|_{m, \gamma+\theta} \lesssim t^{-\theta}\left\|\xi\right\|_{m, \gamma}, \quad \forall t \in \left(0, T\right].
	\end{equation*}
	Combined with the strong continuity of $S$ on $\mathcal{H}_{\gamma+\theta}$ 
	and $\hat{\delta}S\xi=0$, we have 
	$S\xi \in \hat{C}^{1-\theta}L_{m}\mathcal{H}_{\gamma+\theta}\left[t, T\right]$. 
	By Proposition \ref{asp}, applying Minkowski's Inequality, 
	\begin{align*}
		\left\|\hat{\delta}F_{s, t}\right\|_{m, \gamma+\theta} &\lesssim \left(\mathbb{E}\left(\int_{s}^{t}\left|S_{t-r}f\left(r, u_{r}\right)\right|_{\gamma+\theta}dr\right)^{m}\right)^{\frac{1}{m}} \lesssim \int_{s}^{t}\left(t-s\right)^{-\left(\lambda+\theta\right)}\left\|f\left(r, u_{r}\right)\right\|_{m, \gamma-\lambda}dr\\
		&\lesssim \left(\left\|u\right\|_{E^{\beta}L_{m}\mathcal{H}_{\gamma}}+\left\|f\left(\cdot, 0\right)\right\|_{0, m, \gamma-\lambda}\right)\left|t-s\right|^{1-\lambda-\theta}, \quad \forall \left(s, t\right) \in \Delta_{2},
	\end{align*}
	which gives $F \in \hat{C}^{1-\lambda-\theta}L_{m}\mathcal{H}_{\gamma+\theta}$. 
	Similarly, by Proposition \ref{asp}, applying Minkowski's Inequality and BDG Inequality, 
	\begin{align*}
		\left\|\hat{\delta}H_{s, t}\right\|_{m, \gamma+\theta} &\lesssim \left(\mathbb{E}\left(\int_{s}^{t}\left|S_{t-r}h\left(r, u_{r}\right)\right|_{\gamma+\theta}^{2}dr\right)^{\frac{m}{2}}\right)^{\frac{1}{m}} \lesssim \left(\int_{s}^{t}\left(t-r\right)^{-2\left(\mu+\theta\right)}\left\|h\left(r, u_{r}\right)\right\|_{m, \gamma-\mu}^{2}dr\right)^{\frac{1}{2}}\\
		&\lesssim \left(\left\|u\right\|_{E^{\beta}L_{m}\mathcal{H}_{\gamma}}+\left\|h\left(\cdot, 0\right)\right\|_{0, m, \gamma-\mu}\right)\left|t-s\right|^{\frac{1}{2}-\mu-\theta}, \quad \forall \left(s, t\right) \in \Delta_{2},
	\end{align*}
	which gives $H \in \hat{C}^{\frac{1}{2}-\mu-\theta}L_{m}\mathcal{H}_{\gamma+\theta}$. 
	Therefore, 
	$u \in \hat{C}^{\left(1-\lambda\right) \wedge \left(\frac{1}{2}-\mu\right) \wedge \left(\alpha-\nu\right)-\theta}L_{m}\mathcal{H}_{\gamma+\theta}\left[t, T\right]$ 
	for every $t \in \left(0, T\right]$ and we have 
	\begin{align*}
		\left\|u_{t}\right\|_{m, \gamma+\theta} &\leq \left\|S_{t}\xi\right\|_{m, \gamma+\theta}+\left\|\hat{\delta}Z_{0, t}\right\|_{m, \gamma+\theta}+\left\|\hat{\delta}F_{0, t}\right\|_{m, \gamma+\theta}+\left\|\hat{\delta}H_{0, t}\right\|_{m, \gamma+\theta}\\
		&\lesssim t^{-\theta}\left\|\xi\right\|_{m, \gamma}+\left\|u\right\|_{E^{\beta}L_{m}\mathcal{H}_{\gamma}}+\left\|f\left(\cdot, 0\right)\right\|_{0, m, \gamma-\lambda}+\left\|h\left(\cdot, 0\right)\right\|_{0, m, \gamma-\mu}+\left\|g\right\|_{E^{\beta}L_{m}\mathcal{H}_{\gamma-\nu}}.
	\end{align*}
	Combined with \eqref{se}, we get the estimate \eqref{sre}. 
\end{proof}

\section{Continuous mild solutions in $L_{m}E^{\beta}\mathcal{H}_{\gamma}$}
\label{Sec5}
In this section, we will study continuous mild solutions 
to the Young SPDE \eqref{yspde} in $L_{m}E^{\beta}\mathcal{H}_{\gamma}$ 
for given $\beta \in \left(1-\alpha, \frac{1}{2}\right)$ and $m \in \left[2, \infty\right)$. 
\begin{definition} \label{msdp}
	We call $u \in L_{m}E^{\beta}\mathcal{H}_{\gamma}$ a continuous mild solution of \eqref{yspde} if 
	$Gu+g \in L_{m}E^{\beta}\mathcal{H}_{\gamma-\nu}^{e}$ and 
	for a.s. $\omega \in \Omega$, 
	\begin{equation*}
		u_{t}=S_{t}\xi+\int_{0}^{t}S_{t-r}f\left(r, u_{r}\right)dr+\int_{0}^{t}S_{t-r}\left(G_{r}u_{r}+g_{r}\right)d\eta_{r}+\int_{0}^{t}S_{t-r}h\left(r, u_{r}\right)dW_{r}
	\end{equation*}
	holds in $\mathcal{H}_{\gamma}$ for every $t \in \left[0, T\right]$. 
\end{definition}
Note that a continuous mild solution in $L_{m}E^{\beta}\mathcal{H}_{\gamma}$ 
is a mild solution in $E^{\beta}L_{m}\mathcal{H}_{\gamma}$. 
We introduce the following additional assumption. 
\begin{assumption} \label{as2}
	$G \in L_{\infty}E^{\beta}\mathcal{L}_{\gamma, \gamma-\nu}\left(\mathcal{H}, \mathcal{H}^{e}\right)$ and $g \in L_{m}E^{\beta}\mathcal{H}_{\gamma-\nu}^{e}$. 
\end{assumption}
\subsection{Existence and uniqueness}
Similar to Proposition \ref{lc}, we have the following result on compositions. 
\begin{proposition} \label{lcp}
	Let $u \in L_{m}E^{\beta}\mathcal{H}_{\gamma}$ 
	and $G \in L_{\infty}E^{\beta}\mathcal{L}_{\gamma, \gamma-\nu}\left(\mathcal{H}, \mathcal{H}^{e}\right)$. 
	Then $Gu \in L_{m}E^{\beta}\mathcal{H}_{\gamma-\nu}^{e}$ and we have 
	\begin{equation*}
		\left\|Gu\right\|_{L_{m}E^{\beta}\mathcal{H}_{\gamma-\nu}} \lesssim \left\|G\right\|_{L_{\infty}E^{\beta}\mathcal{L}_{\gamma, \gamma-\nu}}\left\|u\right\|_{L_{m}E^{\beta}\mathcal{H}_{\gamma}}.
	\end{equation*}
\end{proposition}
We now give the existence and uniqueness of the continuous mild solution of \eqref{yspde}. 
\begin{theorem} \label{seup}
	Let Assumptions \ref{as} and \ref{as2} hold and 
	$\frac{1}{m} < \left(1-\lambda\right) \wedge \left[\frac{1}{2}-\left(\beta \vee \nu\right)\right]$. 
	Then the Young SPDE \eqref{yspde} has a unique continuous mild solution 
	$u \in L_{m}E^{\beta}\mathcal{H}_{\gamma}$ and we have 
	\begin{equation} \label{sep}
		\left\|u\right\|_{L_{m}E^{\beta}\mathcal{H}_{\gamma}} \lesssim \left\|\xi\right\|_{m, \gamma}+\left\|f\left(\cdot, 0\right)\right\|_{0, m, \gamma-\lambda}+\left\|h\left(\cdot, 0\right)\right\|_{0, m, \gamma-\mu}+\left\|g\right\|_{L_{m}E^{\beta}\mathcal{H}_{\gamma-\nu}},
	\end{equation}
	for a hidden prefactor depending only on $T, \left|\delta \eta\right|_{\alpha}$ and $\left\|G\right\|_{L_{\infty}E^{\beta}\mathcal{L}_{\gamma, \gamma-\nu}}$. 
\end{theorem}
\begin{proof}
	We only show the existence and uniqueness for $T$ sufficiently small. 
	The general result and estimate \eqref{sep} can be obtained analogous to the proof of Theorem \ref{seu}. 
	For any $u \in L_{m}E^{\beta}\mathcal{H}_{\gamma}$, by Proposition \ref{lcp}, 
	$Y:=Gu+g \in L_{m}E^{\beta}\mathcal{H}_{\gamma-\nu}^{e}$. 
	Then by Proposition \ref{ycip}, 
	$Z:=\int_{0}^{\cdot}S_{\cdot-r}Y_{r}d\eta_{r} \in L_{m}E^{\beta}\mathcal{H}_{\gamma}$ and 
	\begin{align*}
		\left\|Z\right\|_{L_{m}E^{\beta}\mathcal{H}_{\gamma}} &\lesssim T^{\alpha-\left(\beta \vee \nu\right)}\left\|Y\right\|_{L_{m}E^{\beta}\mathcal{H}_{\gamma-\nu}}\left|\delta \eta\right|_{\alpha}\\
		&\lesssim T^{\alpha-\left(\beta \vee \nu\right)}\left(\left\|G\right\|_{L_{\infty}E^{\beta}\mathcal{L}_{\gamma, \gamma-\nu}}\left\|u\right\|_{L_{m}E^{\beta}\mathcal{H}_{\gamma}}+\left\|g\right\|_{L_{m}E^{\beta}\mathcal{H}_{\gamma-\nu}}\right)\left|\delta \eta\right|_{\alpha}. 
	\end{align*} 
	Define 
	\begin{equation*}
		F:=\int_{0}^{\cdot}S_{\cdot-r}f\left(r, u_{r}\right)dr, \quad H:=\int_{0}^{\cdot}S_{\cdot-r}h\left(r, u_{r}\right)dW_{r}, \quad \Phi\left(u\right):=S\xi+F+Z+H.
	\end{equation*}
	From the proof of Theorem \ref{seu}, we have $S\xi \in L_{m}E^{\beta}\mathcal{H}_{\gamma}$, 
	$F \in \hat{C}^{1-\left(\lambda-\beta\right)^{+}}L_{m}\mathcal{H}_{\gamma-\beta} \cap \hat{C}^{1-\lambda}L_{m}\mathcal{H}_{\gamma}$ 
	and $H \in \hat{C}^{\frac{1}{2}-\left(\mu-\beta\right)^{+}}L_{m}\mathcal{H}_{\gamma-\beta} \cap \hat{C}^{\frac{1}{2}-\mu}L_{m}\mathcal{H}_{\gamma}$. 
	By Proposition \ref{kc}, $F, H \in L_{m}\hat{C}^{\beta}\mathcal{H}_{\gamma-\beta} \cap L_{m}C\mathcal{H}_{\gamma}$ 
	and there exists $\varepsilon > 0$ depending only on $\beta, \lambda, \mu$ and $m$ such that 
	\begin{equation*}
		\left\|\hat{\delta}F\right\|_{m, \beta, \gamma-\beta} \lesssim T^{\varepsilon}\left\|\hat{\delta}F\right\|_{1-\left(\lambda-\beta\right)^{+}, m, \gamma-\beta}, \quad \left\|F\right\|_{m, 0, \gamma} \lesssim T^{\varepsilon}\left\|\hat{\delta}F\right\|_{1-\lambda, m, \gamma},
	\end{equation*}
	\begin{equation*}
		\left\|\hat{\delta}H\right\|_{m, \beta, \gamma-\beta} \lesssim T^{\varepsilon}\left\|\hat{\delta}H\right\|_{\frac{1}{2}-\left(\mu-\beta\right)^{+}, m, \gamma-\beta}, \quad  \left\|H\right\|_{m, 0, \gamma} \lesssim T^{\varepsilon}\left\|\hat{\delta}H\right\|_{\frac{1}{2}-\mu, m, \gamma}.
	\end{equation*}
	In view of \eqref{Fe} and \eqref{He}, by Proposition \ref{nep}, 
	$F, H \in L_{m}E^{\beta}\mathcal{H}_{\gamma}$ and 
	\begin{equation*}
		\left\|\hat{\delta}F\right\|_{L_{m}E^{\beta}\mathcal{H}_{\gamma}} \lesssim \left\|\hat{\delta}F\right\|_{m, \beta, \gamma-\beta}+\left\|F\right\|_{m, 0, \gamma} \lesssim T^{\varepsilon}\left(\left\|u\right\|_{L_{m}E^{\beta}\mathcal{H}_{\gamma}}+\left\|f\left(\cdot, 0\right)\right\|_{0, m, \gamma-\lambda}\right),
	\end{equation*}
	\begin{equation*}
		\left\|\hat{\delta}H\right\|_{L_{m}E^{\beta}\mathcal{H}_{\gamma}} \lesssim \left\|\hat{\delta}H\right\|_{m, \beta, \gamma-\beta}+\left\|H\right\|_{m, 0, \gamma} \lesssim T^{\varepsilon}\left(\left\|u\right\|_{L_{m}E^{\beta}\mathcal{H}_{\gamma}}+\left\|h\left(\cdot, 0\right)\right\|_{0, m, \gamma-\mu}\right).
	\end{equation*}
	Therefore, $\Phi\left(u\right) \in L_{m}E^{\beta}\mathcal{H}_{\gamma}$ and 
	\begin{align*}
		\left\|\Phi\left(u\right)\right\|_{L_{m}E^{\beta}\mathcal{H}_{\gamma}} &\lesssim T^{\left[\alpha-\left(\beta \vee \nu\right)\right] \wedge \varepsilon}\left(1+\left\|G\right\|_{L_{\infty}E^{\beta}\mathcal{L}_{\gamma, \gamma-\nu}}\left|\delta \eta\right|_{\alpha}\right)\left\|u\right\|_{L_{m}E^{\beta}\mathcal{H}_{\gamma}}\\
		&\quad+\left\|\xi\right\|_{m, \gamma}+\left\|f\left(\cdot, 0\right)\right\|_{0, m, \gamma-\lambda}+\left\|h\left(\cdot, 0\right)\right\|_{0, m, \gamma-\mu}+\left\|g\right\|_{L_{m}E^{\beta}\mathcal{H}_{\gamma-\nu}}\left|\delta \eta\right|_{\alpha}. 
	\end{align*}
	For any other $\bar{u} \in L_{m}E^{\beta}\mathcal{H}_{\gamma}$, we can similarly get 
	\begin{equation*}
		\left\|\Phi\left(u\right)-\Phi\left(\bar{u}\right)\right\|_{L_{m}E^{\beta}\mathcal{H}_{\gamma}} \lesssim T^{\left[\alpha-\left(\beta \vee \nu\right)\right] \wedge \varepsilon}\left(1+\left\|G\right\|_{L_{\infty}E^{\beta}\mathcal{L}_{\gamma, \gamma-\nu}}\left|\delta \eta\right|_{\alpha}\right)\left\|u-\bar{u}\right\|_{L_{m}E^{\beta}\mathcal{H}_{\gamma}}. 
	\end{equation*}
	Hence, $\Phi$ is a contraction map in $L_{m}E^{\beta}\mathcal{H}_{\gamma}$ for $T$ sufficiently small. 
	Applying the Banach Fixed-point Theorem, $\Phi$ has a unique fixed point $u$ in $L_{m}E^{\beta}\mathcal{H}_{\gamma}$, 
	which is the unique continuous mild solution of \eqref{yspde}.
\end{proof}
\subsection{Continuity of the continuous mild solution map}
Similar to Theorem \ref{csm}, we also have the continuity of the continuous mild solution map. 
\begin{theorem} \label{csmp}
	Let Assumptions \ref{as} and \ref{as2} hold, 
	$\frac{1}{m} < \left(1-\lambda\right) \wedge \left[\frac{1}{2}-\left(\beta \vee \nu\right)\right]$ 
	and additionally 
	$\bar{\eta} \in C^{\alpha}\left(\left[0, T\right], \mathbb{R}^{e}\right)$, 
	$\bar{\xi} \in L^{m}\left(\Omega, \mathcal{F}_{0}, \mathcal{H}_{\gamma}\right)$. 
	Assume $\left|\delta \eta\right|_{\alpha}, \left|\delta \bar{\eta}\right|_{\alpha}, \left\|\xi\right\|_{m, \gamma}, \left\|\bar{\xi}\right\|_{m, \gamma} \leq R$ for some $R \geq 0$. 
	Let $u \in L_{m}E^{\beta}\mathcal{H}_{\gamma}$ be the continuous mild solution of \eqref{yspde} and 
	$\bar{u} \in L_{m}E^{\beta}\mathcal{H}_{\gamma}$ be the continuous mild solution of \eqref{yspde} with $\eta$ and $\xi$ 
	replaced by $\bar{\eta}$ and $\bar{\xi}$, respectively. 
	Then we have 
	\begin{equation*}
		\left\|u-\bar{u}\right\|_{L_{m}E^{\beta}\mathcal{H}_{\gamma}} \lesssim \left|\delta \eta-\delta \bar{\eta}\right|_{\alpha}+\left\|\xi-\bar{\xi}\right\|_{m, \gamma}, 
	\end{equation*}
	for a hidden prefactor depending only on $T, R, \left\|f\left(\cdot, 0\right)\right\|_{0, m, \gamma-\lambda}, \left\|h\left(\cdot, 0\right)\right\|_{0, m, \gamma-\mu}, \left\|G\right\|_{L_{\infty}E^{\beta}\mathcal{L}_{\gamma, \gamma-\nu}}$ and $\left\|g\right\|_{L_{m}E^{\beta}\mathcal{H}_{\gamma-\nu}}$. 
	As a consequence, the continuous mild solution map 
	\begin{equation*}
		C^{\alpha}\left(\left[0, T\right], \mathbb{R}^{e}\right) \times L^{m}\left(\Omega, \mathcal{F}_{0}, \mathcal{H}_{\gamma}\right) \rightarrow L_{m}E^{\beta}\mathcal{H}_{\gamma}: \quad \left(\eta, \xi\right) \mapsto u
	\end{equation*}
	is locally Lipschitz continuous. 
\end{theorem}

\section{An example}
\label{Sec6}
Let $\mathbb{T}^{n}$ be the $n$-dimensional torus, 
$\mathcal{H}:=L^{2}\left(\mathbb{T}^{n}, \mathbb{R}^{l}\right)$ and 
$L:=\Delta$ be the Laplace operator on $\mathbb{T}^{n}$. 
Define $H^{\gamma}\left(\mathbb{T}^{n}, \mathbb{R}^{l}\right)$ as the 
$L^{2}$-based fractional order Sobolev space. 
Then $L$ is a negative definite self-adjoint operator on $\mathcal{H}$ with 
domain $\mathcal{D}\left(L\right):=H^{2}\left(\mathbb{T}^{n}, \mathbb{R}^{l}\right)$,  
which generates the heat semigroup $\left(S_{t}\right)_{t \geq 0}:=\left(e^{t\Delta}\right)_{t \geq 0}$. 
Furthermore, $\mathcal{H}_{\gamma}=H^{2\gamma}\left(\mathbb{T}^{n}, \mathbb{R}^{l}\right)$ 
for every $\gamma \in \mathbb{R}$.\\
\indent
Consider the following concrete Young SPDE
\begin{equation} \label{yspdec}
	\left\{
		\begin{aligned}
			&du_{t}\left(x\right)=\left[\Delta u_{t}\left(x\right)+f\left(t, x, u_{t}\left(x\right), \nabla u_{t}\left(x\right)\right)\right]dt+\left[G_{t}^{1}\left(x\right)\nabla u_{t}\left(x\right)+G_{t}^{0}\left(x\right)u_{t}\left(x\right)+g_{t}\left(x\right)\right]d\eta_{t}\\
			&\quad \quad \quad \quad+h\left(t, x, u_{t}\left(x\right)\right)dW_{t}, \quad \left(t, x\right) \in \left(0, T\right] \times \mathbb{T}^{n},\\
			&u_{0}\left(x\right)=\xi\left(x\right), \quad x \in \mathbb{T}^{n}.
		\end{aligned}
	\right.	
\end{equation}
Here, $\eta \in C^{\alpha}\left(\left[0, T\right], \mathbb{R}^{e}\right)$, 
$\xi: \Omega \times \mathbb{T}^{n} \rightarrow \mathbb{R}^{l}$ is the initial datum, 
$f: \left[0, T\right] \times \Omega \times \mathbb{T}^{n} \times \mathbb{R}^{l} \times \mathbb{R}^{l \times n} \rightarrow \mathbb{R}^{l}$, 
$G^{1}: \left[0, T\right] \times \Omega \times \mathbb{T}^{n} \rightarrow \mathcal{L}\left(\mathbb{R}^{l \times n}, \mathbb{R}^{l \times e}\right)$, 
$G^{0}: \left[0, T\right] \times \Omega \times \mathbb{T}^{n} \rightarrow \mathcal{L}\left(\mathbb{R}^{l}, \mathbb{R}^{l \times e}\right)$, 
$g: \left[0, T\right] \times \Omega \times \mathbb{T}^{n} \rightarrow \mathbb{R}^{l \times e}$ 
and $h: \left[0, T\right] \times \Omega \times \mathbb{T}^{n} \times \mathbb{R}^{l} \rightarrow \mathbb{R}^{l \times d}$ 
are progressively measurable vector fields. 
Given $\beta \in \left(1-\alpha, \frac{1}{2}\right)$ and $m \in \left[2, \infty\right)$, 
we introduce the following definition of mild solutions and continuous mild solutions 
as Definitions \ref{msd} and \ref{msdp} for $\gamma=0$.  
\begin{definition} \label{msdc}
	We call $u \in C^{\beta}L_{m}H^{-2\beta}\left(\mathbb{T}^{n}, \mathbb{R}^{l}\right) \cap CL_{m}L^{2}\left(\mathbb{T}^{n}, \mathbb{R}^{l}\right)$ a mild solution of \eqref{yspdec} if 
	$G^{1}\nabla u+G^{0}u+g \in C^{\beta}L_{m}H^{-1-2\beta}\left(\mathbb{T}^{n}, \mathbb{R}^{l \times e}\right) \cap CL_{m}H^{-1}\left(\mathbb{T}^{n}, \mathbb{R}^{l \times e}\right)$ and 
	for every $t \in \left[0, T\right]$ and a.s. $\omega \in \Omega$, 
	\begin{equation*}
		u_{t}=S_{t}\xi+\int_{0}^{t}S_{t-r}f\left(r, u_{r}, \nabla u_{r}\right)dr+\int_{0}^{t}S_{t-r}\left(G_{r}^{1}\nabla u_{r}+G_{r}^{0}u_{r}+g_{r}\right)d\eta_{r}+\int_{0}^{t}S_{t-r}h\left(r, u_{r}\right)dW_{r}
	\end{equation*}
	holds in $L^{2}\left(\mathbb{T}^{n}, \mathbb{R}^{l}\right)$. 
\end{definition}
\begin{definition} \label{msdpc}
	We call $u \in L_{m}C^{\beta}H^{-2\beta}\left(\mathbb{T}^{n}, \mathbb{R}^{l}\right) \cap L_{m}CL^{2}\left(\mathbb{T}^{n}, \mathbb{R}^{l}\right)$ a continuous mild solution of \eqref{yspdec} if 
	$G^{1}\nabla u+G^{0}u+g \in L_{m}C^{\beta}H^{-1-2\beta}\left(\mathbb{T}^{n}, \mathbb{R}^{l \times e}\right) \cap L_{m}CH^{-1}\left(\mathbb{T}^{n}, \mathbb{R}^{l \times e}\right)$ and 
	for a.s. $\omega \in \Omega$, 
	\begin{equation*}
		u_{t}=S_{t}\xi+\int_{0}^{t}S_{t-r}f\left(r, u_{r}, \nabla u_{r}\right)dr+\int_{0}^{t}S_{t-r}\left(G_{r}^{1}\nabla u_{r}+G_{r}^{0}u_{r}+g_{r}\right)d\eta_{r}+\int_{0}^{t}S_{t-r}h\left(r, u_{r}\right)dW_{r}
	\end{equation*}
	holds in $L^{2}\left(\mathbb{T}^{n}, \mathbb{R}^{l}\right)$ for every $t \in \left[0, T\right]$. 
\end{definition}
Then we introduce the following assumptions. 
\begin{assumption} \label{asc}
	~
	\begin{enumerate}[(i)]
		\item $\xi \in L^{m}\left(\Omega, \mathcal{F}_{0}, L^{2}\left(\mathbb{T}^{n}, \mathbb{R}^{l}\right)\right)$; 
		\item $t \mapsto f\left(t, \cdot, 0, 0\right)$ is bounded from $\left[0, T\right]$ to $L^{m}\left(\Omega, L^{2}\left(\mathbb{T}^{n}, \mathbb{R}^{l}\right)\right)$ 
		and for every $\left(t, x\right) \in \left[0, T\right] \times \mathbb{T}^{n}$, $u, \bar{u} \in \mathbb{R}^{l}$ and $v, \bar{v} \in \mathbb{R}^{l \times n}$, 
		\begin{equation*}
			\left|f\left(t, x, u, v\right)-f\left(t, x, \bar{u}, \bar{v}\right)\right| \lesssim \left|u-\bar{u}\right|+\left|v-\bar{v}\right|;
		\end{equation*}
		\item $G^{1} \in C^{\beta}L_{m}L^{\infty}\left(\mathbb{T}^{n}, \mathcal{L}\left(\mathbb{R}^{l \times n}, \mathbb{R}^{l \times e}\right)\right)$ 
		and $G^{0} \in C^{\beta}L_{m}L^{\infty}\left(\mathbb{T}^{n}, \mathcal{L}\left(\mathbb{R}^{l}, \mathbb{R}^{l \times e}\right)\right)$;
		\item $g \in C^{\beta}L_{m}H^{-1-2\beta}\left(\mathbb{T}^{n}, \mathbb{R}^{l \times e}\right) \cap CL_{m}H^{-1}\left(\mathbb{T}^{n}, \mathbb{R}^{l \times e}\right)$;
		\item $t \mapsto h\left(t, \cdot, 0\right)$ is bounded from $\left[0, T\right]$ to $L^{m}\left(\Omega, L^{2}\left(\mathbb{T}^{n}, \mathbb{R}^{l \times d}\right)\right)$ and 
		\begin{equation*}
			\left|h\left(t, x, u\right)-h\left(t, x, \bar{u}\right)\right| \lesssim \left|u-\bar{u}\right|, \quad \forall \left(t, x\right) \in \left[0, T\right] \times \mathbb{T}^{n}, \quad \forall u, \bar{u} \in \mathbb{R}^{l}. 
		\end{equation*}
	\end{enumerate}
\end{assumption}
\begin{assumption} \label{as2c}
	$G^{1} \in L_{m}C^{\beta}L^{\infty}\left(\mathbb{T}^{n}, \mathcal{L}\left(\mathbb{R}^{l \times n}, \mathbb{R}^{l \times e}\right)\right)$, 
	$G^{0} \in L_{m}C^{\beta}L^{\infty}\left(\mathbb{T}^{n}, \mathcal{L}\left(\mathbb{R}^{l}, \mathbb{R}^{l \times e}\right)\right)$ 
	and $g \in L_{m}C^{\beta}H^{-1-2\beta}\left(\mathbb{T}^{n}, \mathbb{R}^{l \times e}\right) \cap L_{m}CH^{-1}\left(\mathbb{T}^{n}, \mathbb{R}^{l \times e}\right)$.
\end{assumption}
Clearly, Assumption \ref{asc} implies Assumption \ref{as} 
for $\gamma=0, \lambda=\nu=\frac{1}{2}$ and $\mu=0$. 
By Theorems \ref{seu} and \ref{csm} and Proposition \ref{sr}, we have 
the following result. 
\begin{theorem} \label{ms}
	Under Assumption \ref{asc}, the concrete Young SPDE \eqref{yspdec} has a unique mild solution 
	$u \in C^{\beta}L_{m}H^{-2\beta}\left(\mathbb{T}^{n}, \mathbb{R}^{l}\right) \cap CL_{m}L^{2}\left(\mathbb{T}^{n}, \mathbb{R}^{l}\right)$ 
	and the mild solution map 
	\begin{equation*}
		C^{\alpha}\left(\left[0, T\right], \mathbb{R}^{e}\right) \times L^{m}\left(\Omega, \mathcal{F}_{0}, L^{2}\left(\mathbb{T}^{n}, \mathbb{R}^{l}\right)\right) \rightarrow C^{\beta}L_{m}H^{-2\beta}\left(\mathbb{T}^{n}, \mathbb{R}^{l}\right) \cap CL_{m}L^{2}\left(\mathbb{T}^{n}, \mathbb{R}^{l}\right): \quad \left(\eta, \xi\right) \mapsto u
	\end{equation*}
	is locally Lipschitz continuous. 
	Furthermore, $u_{t} \in L^{m}\left(\Omega, H^{2\theta}\left(\mathbb{T}^{n}, \mathbb{R}^{l}\right)\right)$ 
	for every $t \in \left(0, T\right]$ and $\theta \in \left[0, \alpha-\frac{1}{2}\right)$. 
\end{theorem}
Similarly, Assumption \ref{as2c} implies Assumption \ref{as2} 
for $\gamma=0, \lambda=\nu=\frac{1}{2}$ and $\mu=0$. 
By Theorems \ref{seup} and \ref{csmp}, we have the following futher result. 
\begin{theorem}
	Let Assumptions \ref{asc} and \ref{as2c} hold and $\beta+\frac{1}{m} < \frac{1}{2}$.  
	Then the concrete Young SPDE \eqref{yspdec} has a unique continuous mild solution 
	$u \in L_{m}C^{\beta}H^{-2\beta}\left(\mathbb{T}^{n}, \mathbb{R}^{l}\right) \cap L_{m}CL^{2}\left(\mathbb{T}^{n}, \mathbb{R}^{l}\right)$ 
	and the continuous mild solution map 
	\begin{equation*}
		C^{\alpha}\left(\left[0, T\right], \mathbb{R}^{e}\right) \times L^{m}\left(\Omega, \mathcal{F}_{0}, L^{2}\left(\mathbb{T}^{n}, \mathbb{R}^{l}\right)\right) \rightarrow L_{m}C^{\beta}H^{-2\beta}\left(\mathbb{T}^{n}, \mathbb{R}^{l}\right) \cap L_{m}CL^{2}\left(\mathbb{T}^{n}, \mathbb{R}^{l}\right): \quad \left(\eta, \xi\right) \mapsto u
	\end{equation*}
	is locally Lipschitz continuous. 
\end{theorem}

\bibliographystyle{amsplain}
\bibsection

\end{document}